\documentclass[times,a4paper,doublespace]{amsart}

\usepackage{graphics,epic}
\usepackage{amsmath,amssymb, amsthm,mathrsfs}
\usepackage[all,2cell]{xy}
\usepackage{tikz}
\usetikzlibrary{arrows}

\newtheorem{theorem}{Theorem}[section]
\newtheorem*{theorem*}{Theorem}

\newtheorem{lemma}[theorem]{Lemma}
\newtheorem{proposition}[theorem]{Proposition}
\newtheorem{corollary}[theorem]{Corollary}
\newtheorem{conjecture}[theorem]{Conjecture}
\newtheorem*{conjecture*}{Conjecture}

\newtheorem{definition}[theorem]{Definition}

\newcommand{\opname}[1]{\operatorname{\mathsf{#1}}}

\renewcommand{\mod}{\opname{mod}\nolimits}

\newcommand{\add}{\opname{add}\nolimits}
\newcommand{\op}{^{op}}
\newcommand{\der}{\cd}

\newcommand{\dimv}{\underline{\dim}\,}

\newcommand{\ind}{\opname{ind}}

\newcommand{\rad}{\opname{rad}\nolimits}

\newcommand{\Z}{\mathbb{Z}}
\newcommand{\N}{\mathbb{N}}

\renewcommand{\P}{\mathbb{P}}

\newcommand{\Hom}{\opname{Hom}}

\newcommand{\Ext}{\opname{Ext}}

\newcommand{\End}{\opname{End}}

\newcommand{\cd}{{\mathcal D}}

\newcommand{\ct}{{\mathcal T}}

\newcommand{\cw}{{\mathcal W}}

\newcommand{\mc}{\mathcal{C}}

\newcommand{\mf}{\mathcal{F}}

\newcommand{\ma}{\mathcal{A}}
\newcommand{\mb}{\mathcal{B}}
\newcommand{\mt}{\mathcal{T}}

\newcommand{\mz}{\mathcal{Z}}

\renewcommand{\hat}[1]{\widehat{#1}}

\begin{document}

\title[denominator vectors of tame type]{On indecomposable $\tau$-rigid modules for cluster-tilted algebras of tame type}

\author{Changjian Fu}
\address{Changjian Fu\\Department of Mathematics\\SiChuan University\\610064 Chengdu\\P.R.China}
\email{changjianfu@scu.edu.cn}
\author{Shengfei Geng}
\address{Shengfei Geng\\Department of Mathematics\\SiChuan University\\610064 Chengdu\\P.R.China}
\email{genshengfei@scu.edu.cn}
\subjclass[2010]{16G20, 13F60}
\keywords{cluster category, cluster-tilted algebra, $\tau$-rigid module, cluster algebra of tame type, denominator conjecture}

\begin{abstract}
For a given cluster-tilted algebra $A$ of tame type, it is proved that different indecomposable $\tau$-rigid $A$-modules have different dimension vectors. This is motivated by Fomin and Zelevinsky's denominator conjecture for cluster algebras. As an application, we establish a weak version of the denominator conjecture for cluster algebras of tame type. Namely, we show that different cluster variables have different denominators with respect to a given cluster for a cluster algebra of tame type. Our approach involves Iyama and Yoshino's construction of subfactors of triangulated categories. In particular, we obtain a description of the subfactors of cluster categories of tame type with respect to an indecomposable rigid object, which is of independent interest.
\end{abstract}

\maketitle

\section{Introduction}
\subsection{Motivation}
Cluster algebras were invented by Fomin and Zelevinsky~\cite{FZ02} with the purpose to provide an algebraic framework for the study of total positivity in algebraic groups and canonical bases in quantum groups. They are commutative algebras defined via a set of generators: {\it cluster variables}, constructed inductively by mutations. The cluster variables were gathered into overlapping sets of fixed cardinality called {\it clusters}. Monomials of cluster variables all of which belong to the same cluster are {\it cluster monomials}.
Inspired by Lusztig's parameterization of canonical bases in the theory of quantum groups, Fomin and Zelevinsky~\cite{FZ03-2} introduced the {\it denominator vector} for each cluster monomial with respect to a given cluster and formulated the so-called {\bf denominator conjecture}: {\it different cluster monomials have different denominator vectors}. In contrast to other conjectures of cluster algebras, little progress has been made towards proving the denominator conjecture. One of the reason is that there does not exist a simple formula relating the denominator vectors of a cluster variable with respect to different seeds. To our best knowledge, the denominator conjecture has not yet been verified for cluster algebras of finite type.

In order to understand the denominator conjecture, one would first to consider the following weak version: {\it different cluster variables have different denominator vectors with respect to a given cluster}. Nevertheless, the weak version has only been verified for very restricted cases, see~\cite{SZ, CK06,CK08,GP, AD} for instance. For example, it has been established by Geng and Peng~\cite{GP} for skew-symmetric cluster algebras of finite type and by Caldero and Keller~\cite{CK06,CK08} for acyclic cluster algebras with respect to acyclic initial seeds\footnote{Using the representation theory of hereditary algebras, Caldero and Keller's work implies the denominator conjecture for acyclic cluster algebras with respect to acyclic initial seeds.}. For cluster algebras of rank ~$2$, it has been verified by Sherman and Zelevinsky~\cite{SZ}. By the work of Nakanishi and Stella~\cite{NS}, we also know that the weak denominator conjecture holds for cluster algebras of finite type.

One of the key ingredient in the work of Caldero and Keller~\cite{CK06,CK08} is that there exist categorifications for acyclic cluster algebras.
More precisely, let $Q$ be an acyclic quiver and $\ma(Q)$ the acyclic cluster algebra associated to $Q$.
Denote by $\mc_Q$ the cluster category~\cite{BMRRT} associated to $Q$, which is a $2$-Calabi-Yau triangulated category with cluster-tilting objects.  Let $\Sigma$ be the suspension functor of $\mc_Q$. It has been proved in~\cite{CK06} that there is a bijection between the cluster variables of $\ma(Q)$ and the indecomposable rigid objects in $\mc_Q$. Moreover, the clusters of $\ma(Q)$ correspond to the basic cluster-tilting objects of $\mc_Q$. We fix a cluster $Y$ of $\ma(Q)$ as initial seed, let $\Sigma T$ be the corresponding basic cluster-tilting object in~$\mc_Q$. Denote by $\Gamma=\End_{\mc_Q}(T)$ the cluster-tilted algebra~\cite{BMR07} associated to $T$ and by $\tau$ the Auslander-Reiten translation of $\Gamma$-modules. There is a bijection between the non-initial cluster variables of $\ma(Q)$  and the indecomposable $\tau$-rigid $\Gamma$-modules~\cite{AIR}. If $Y$ is an acyclic seed, Caldero and Keller~\cite{CK06} proved that the denominator vector of a non-initial cluster variable is precisely the dimension vector of the corresponding indecomposable $\tau$-rigid $\Gamma$-modules.
However, such an interpretation is no longer true even for acyclic cluster algebras of tame type with respect to non-acyclic initial seeds~({\it cf.}~\cite{CI, FK, BMR09, BM10,MR}). Nevertheless, for cluster algebras of tame type, Buan and Marsh~\cite{BM10}({\it cf.} also~\cite{BMR09}) established an explicitly relation between the denominator vectors of non-initial cluster variables and the dimension vectors of indecomposable $\tau$-rigid modules over corresponding cluster-tilted algebras.
\subsection{Main results}
The main purpose of this paper is to establish the weak denominator conjecture for cluster algebras of tame type by using the explicitly relation obtained in~\cite{BM10}. Let $Q$ be a connected extended Dynkin quiver and $\ma(Q)$ the associated cluster algebra of tame type. Denote by $\mc_Q$ the cluster category of $Q$. Note that the Auslander-Reiten quiver of $\mc_Q$ consists of a connected component with shape $\mathbb{Z}Q$ (objects in this component are called {\it transjective}), infinitely many homogeneous tubes and finitely many non-homogeneous tubes ({\it i.e.} tube with rank strictly greater than $1$). If $d_1,\cdots, d_r$ are precisely the ranks of non-homogeneous tubes, we also call $\mc_Q$ a {\it cluster category of tame type $(d_1,\cdots, d_r)$}.
For a basic cluster-tilting object $T\in \mc_Q$,  $\Gamma=\End_{\mc_Q}(T)^{op}$ is called a {\it cluster-tilted algebra of tame type}. Recall
that there is a bijection between the non-initial cluster variables of $\ma(Q)$ and the indecomposable $\tau$-rigid $\Gamma$-modules. The first main result of this paper is an analogue of the weak denominator conjecture, which answers a question of~\cite{BMR09} for cluster-tilted algebras of tame type.
\begin{theorem}[Theorem~\ref{t:main-theorem-dimension-vector}]~\label{t:main-result-1}
Let $\Gamma$ be a cluster-tilted algebra of tame type. Then different indecomposable $\tau$-rigid $\Gamma$-modules have different dimension vectors.
\end{theorem}
We remark that the analogue results have been established for cluster-tilted algebras of finite type in~\cite{GP, Ri11} and for cluster-concealed algebras in~\cite{AD,G}.  Zhang~\cite{Zh} also proved that all the indecomposable $\tau$-rigid modules over a gentle algebra arising from an unpunctured surface are determined by their dimension vectors. 

By applying Theorem~\ref{t:main-result-1} and the main result of~\cite{BM10}, we prove the second main result of this paper in Section~\ref{S:denominator}. Namely,
\begin{theorem}[Theorem~\ref{t:denominator-conjecture-tame-type}]
Let $Q$ be a connected extended Dynkin quiver and $\ma(Q)$ the associated cluster algebra. Then different cluster variables of $\ma(Q)$ have different denominator vectors with respect to a given cluster.
\end{theorem}
Our proof of Theorem~\ref{t:main-result-1} relies on the investigation of Iyama and Yoshino's construction of subfactors for cluster categories. In particular, we obtain a description of subfactor of cluster categories of tame type with respect to an indecomposable rigid object, which is of independent interest. To state the result, let us first recall some more notation. For a cluster category $\mc$ with suspension functor~$\Sigma$ and a rigid object $Z\in \mc$, denote by $\add Z$ the subcategory of $\mc$ consisting of objects which are finite direct sum of direct summands of $Z$. Set
\[~^\perp(\add \Sigma Z)=\{X\in \mc~|~\Hom_\mc(X, \Sigma M)=0~\text{for~$\forall~M\in \add Z$}\}.
\]
It has been proved in~\cite{IY} that the subfactor $~^\perp(\add\Sigma Z)/\add Z$ inherits a triangle structure from $\mc$. The following result suggests that cluster categories are closed under subfactors, which seems to be known for experts.
\begin{theorem}[Theorem~\ref{t:subfactor-cluster-category}]
 Let $\mc$ be a cluster category and $Z$ a rigid object of $\mc$. The subfactor $^\perp(\add\Sigma Z)/\add Z$ is a cluster category.
\end{theorem}

Let $\der$ be a triangulated category with triangulated subcategories $\der_1$ and $\der_2$. We say that $\der$ is a {\it direct sum of $\der_1$ and $\der_2$}, provided that:
\begin{itemize}
\item[$\bullet$] any object $M\in \der$ is a direct sum of objects $M_1\in \der_1$ and $M_2\in \der_2$;
\item[$\bullet$] $\Hom_\der(\der_1, \der_2)=0=\Hom_{\der}(\der_2,\der_1)$.
\end{itemize} 
\begin{theorem}[Theorem~\ref{t:subfactor-tame-type}]
Let $\mc$ be a cluster category of tame type $(d_1,\cdots, d_r)$ and $Z$ an indecomposable rigid object of $\mc$. Denote by $\mc':=~^\perp(\add\Sigma Z)/\add Z$. We have
\begin{itemize}
\item[(i)] if $Z$ is transjective, then $\mc'$ is a cluster category of finite type;
\item[(ii)] if $Z$ lies in the tube of rank ~$d_l$ with quasi-length $q.l.(Z)=t<d_l-1$, then $\mc'$ is the direct sum of a cluster category of tame type $(d_1,\cdots, d_{l-1}, d_l-t, d_{l+1}, \cdots, d_r)$ with a cluster category of Dynkin quiver of type $A_{t-1}$.
\end{itemize}
\end{theorem}
The paper is structured as follows. In Section~\ref{S:cluster-tilting-theorey}, we recall basic definitions and facts for cluster-tilting theory. Section~\ref{S:subfactor} is devoted to investigate the subfactors of cluster categories. In particular, Theorem~\ref{t:subfactor-cluster-category} and Theorem~\ref{t:subfactor-tame-type} are proved. We then apply results obtained in Section~\ref{S:subfactor} to prove Theorem~\ref{t:main-theorem-dimension-vector} in Section~\ref{S:dimension-vector}. Finally, we prove the weak denominator conjecture (Theorem~\ref{t:denominator-conjecture-tame-type}) for cluster algebras of tame type in Section~\ref{S:denominator}.

\subsection{Convention}
Let $k$ be an algebraically closed field. Denote by $D=\Hom_k(-,k)$ the
 duality over $k$.
For an object $M$ in a category $\mc$, denote by $|M|$ the number of non-isomorphic indecomposable direct summands of $M$ and by $\add M$ the subcategory of $\mc$ consisting of objects which are finite direct sum of direct summands of $M$. For $r\in \N$, we denote by $M^{\oplus r}$ the direct sum of $r$ copies of $M$.
 For a triangulated category $\mc$, we always denote by $\Sigma$ the suspension functor of $\mc$ and we refer to~\cite{Ha} for the basic properties of triangulated categories. For basic facts on representation theory of finite-dimensional hereditary algebras, we refer to~\cite{Ri84,ASS,SS,SS2}.

\noindent{\bf Acknowledgements.} This work was partially supported by the National Natural Science Foundation of China (No. 11471224). This work was done during the stay of the second-named author at the Department of Mathematics, University of Bielefeld, supported by the China Scholarship Council. S. Geng is deeply indebted to Professor Henning Krause for his kind hospitality and encouragement.

\section{Recollection on cluster-tilting theory}~\label{S:cluster-tilting-theorey}

\subsection{Cluster structure for $2$-Calabi-Yau categories}~\label{s:cluster-structure}
Let $\mc$ be a $2$-Calabi-Yau triangulated category with suspension functor $\Sigma$. In particular, for any $X,Y\in \mc$, we have the following bifunctorially isomorphism
\[\Hom_{\mc}(X,Y)\cong D\Hom_\mc(Y,\Sigma^2 X).
\]
An object $M\in \mc$ is {\it rigid} provided $\Hom_\mc(M,\Sigma M)=0$. A full subcategory $\mathcal{N}$ of $\mc$ is {\it rigid} if $\Hom_{\mc}(X,\Sigma Y)=0$ for any  $X,Y\in \mathcal{N}$.
Recall that a subcategory $\mathcal{N}$ is {\it contravariantly finite} in $\mc$, if any object $M\in \mc$ admits a right $\mathcal{N}$-approximation $f : N\to M$, which means that any map from $N'\in \mathcal{N}$ to $M$ factors through $f$. The left $\mathcal{N}$-approximation of $M$ and {\it covariantly finiteness} of $\mathcal{N}$ can be defined dually. A subcategory $\mathcal{N}$ is called {\it functorially finite} in $\mc$ if $\mathcal{N}$ is both covariantly finite and contravariantly finite in $\mc$.
A full functorially finite subcategory $\mathcal{N}$ of $\mc$ is a {\it cluster-tilting subcategory} if
\begin{itemize}
\item[$\bullet$] $\mathcal{N}$ is a rigid subcategory;
\item[$\bullet$] if an object $M\in \mc$ such that $\Hom_{\mc}(N, \Sigma M)=0$ for any $N\in \mathcal{N}$, then we have $M\in \mathcal{N}$.
\end{itemize}
An object $T\in \mc$ is a {\it cluster-tilting object} if the subcategory $\add T$ is a cluster-tilting subcategory of $\mc$. A typical class of $2$-Calabi-Yau triangulated categories with cluster-tilting objects is the cluster categories introduced in~\cite{BMRRT} ({\it cf.}~also  Section~\ref{ss:cluster-category}).

Let $\mc$ be a $2$-Calabi-Yau triangulated category with cluster-tilting objects. Let $T=\overline{T}\oplus T_k$ be a basic cluster-tilting object with indecomposable direct summand $T_k$. It has been shown in~\cite{IY} that there is a unique indecomposable object $T_k^*$ such that $T_k\not\cong T_k^*$ and $\mu_{T_k}(T):=\overline{T}\oplus T_k^*$ is a basic cluster-tilting object.
Moreover, $T_k^*$ is uniquely determined by the so-called {\it exchange triangles}
\[T_k\xrightarrow{f}B\xrightarrow{g}T_k^*\to \Sigma T_k~\text{and}~T_k^*\xrightarrow{f'}B'\xrightarrow{g'}T_k\to \Sigma T_k^*,
\]
 where $g, g'$ are minimal right $\add \overline{T}$-approximations. In this case, the cluster-tilting object $\mu_{T_k}(T)$ is called the {\it mutation of $T$ at the indecomposable direct summand $T_k$} and the pair $(T_k, T_k^*)$ is an {\it exchange pair}.

  For a given basic cluster-tilting object $T\in \mc$, denote by $Q_T$ the quiver of $T$. By definition, its vertices correspond to the indecomposable direct summands of $T$ and the arrows from the indecomposable direct summand $T_i$ to $T_j$ is given by the dimension of the space of irreducible maps $\rad(T_i, T_j)/\rad^2(T_i, T_j)$, where $\rad(-,-)$ is the radical of the category $\add T$. The category $\mc$ has {\it no loops nor $2$-cycles} provided that for each basic cluster-tilting object $T\in \mc$, its quiver $Q_T$ has no loops nor $2$-cycles.

Let $Q$ be a finite quiver with vertex set $Q_0=\{1,\cdots,n\}$. If $Q$ has no loops nor $2$-cycles, we define a skew-symmetric matrix $B(Q)=(b_{ij})\in M_n(\Z)$, where
\[b_{ij}=\begin{cases}0& i=j;\\ |\{\text{arrows~ $i\to j$}\}|-|\{\text{arrows $j\to i$}\}|& i\neq j.\end{cases}
\]
On the other hand, for a given skew-symmetric integer matrix $B$, one can construct a quiver $Q$ without loops nor $2$-cycles such that $B=B(Q)$.
For any $1\leq k\leq n$, the {\it Fomin-Zelevinsky's mutation}~\cite{FZ02} transforms the matrix $B(Q)$ to a new skew-symmetric matrix $\mu_k(B(Q))=(b_{ij}')\in M_n(\Z)$, where
\[b_{ij}'=\begin{cases}-b_{ij}&\text{$i=k$~ or ~$j=k$;}\\ b_{ij}+\frac{|b_{ik}|b_{kj}+b_{ik}|b_{kj}|}{2}& \text{else}.\end{cases}
\]

Let $\mc$ be a $2$-Calabi-Yau triangulated category with cluster-tilting objects. The category $\mc$ admits a {\it cluster structure}~\cite{BIRS} if the following conditions are satisfied:
\begin{itemize}
\item[(C1)] $\mc$ has no loops nor $2$-cycles;
\item[(C2)] For each basic cluster-tilting object $T$ with an indecomposable direct summand $T_k$, the matrix $B(Q_T)$ and $B(Q_{\mu_{T_k}(T)})$ are related by the Fomin-Zelevinsky's mutation rule.
\end{itemize}
It has been proved in~\cite{BIRS} that the condition ~(C1) implies (C2). Namely, we have
\begin{theorem}~\label{t:cluster-structure}
Let $\mc$ be a $2$-Calabi-Yau triangulated category with cluster-tilting objects. If $\mc$ has no loops nor $2$-cycles, then $\mc$ has a cluster structure.
\end{theorem}

\subsection{Recollection on hereditary algebras}
We recall some facts about hereditary algebras following~\cite{Ri84, ASS,SS,SS2}. 

Let $Q$ be a finite acyclic quiver and $kQ$ the path algebra of $Q$ over $k$. We assume moreover that $Q$ is connected.
Denote by $\mod kQ$ the category of finitely generated left $kQ$-modules.
Let $\tau$ be the Auslander-Reiten translation of $\mod kQ$. An indecomposable $kQ$-module $M$ is {\it preprojective}  if $\tau^m M$ is projective for some $m\geq 0$; $M$ is  {\it preinjective} if $\tau^{-m}M$ is injective for some $m\geq 0$; $M$ is {\it regular} if $M$ is neither preprojective nor preinjective. The Auslander-Reiten (AR for short) quiver of $\mod kQ$ has been well-understood ({\it cf.}~\cite{Ri84,ASS, SS, SS2} for instance).
\begin{enumerate}
\item[(i)] If $Q$ is of Dynkin type, there are only finitely many indecomposable $kQ$-modules up to isomorphism and each indecomposable module is preprojective and preinjective. In particular, the AR quiver of $\mod kQ$ consists of a unique connected component;

\item[(ii)] If $Q$ is of extended Dynkin type, the AR quiver of $\mod kQ$ consists of :
\begin{itemize}
\item[$\circ$] the preprojective component, consisting exactly of the indecomposable preprojective modules;
\item[$\circ$] the preinjective component, consisting exactly of the indecomposable preinjective modules;
\item[$\circ$] a finite number of regular components called  non-homogeneous tubes;
\item[$\circ$] an infinite set of regular components called  homogeneous tubes;
\end{itemize}

\item[(iii)] If $Q$ is of wild type, the AR quiver of $\mod kQ$ consists of the preprojective component, the preinjective component and infinitely many connected regular components with shape $\Z\mathbb{A}_\infty$.
\end{enumerate}

Now assume that $Q$ is a connected extended Dynkin quiver. We turn to recall some facts on the regular components of $\mod kQ$.

For a tube $\mt$, we also denote by $\mt$ the full subcategory of $\mod kQ$ consisting of objects  which are finite direct sum of indecomposable objects lying in $\mt$. Notice that each tube $\mt$ is an abelian subcategory of $\mod kQ$ and  the Auslander-Reiten translation $\tau$ restricts to an autoequivalence of $\mt$. Moreover,  $\mt$ is standard, {\it i.e.} the subcategory of $\mt$ consisting of the indecomposable objects is equivalent to the mesh category of the AR quiver of $\mt$.
For a fixed tube $\mt$, there is a positive integer $m$, such that $\tau^m M=M$ for all indecomposable objects in $\ct$. The minimal such $m$ is the {\it rank} of $\mt$. If $m = 1$, then $\mt$ is said to be homogeneous.

Let $\mt$ be a tube of  rank  $d$ in $\mod kQ$. 
Indecomposable modules lying in the mouth of $\mt$ are {\it quasi-simple} modules. 
There are exactly $d$ non-isomorphic quasi-simple $kQ$-modules in~$\mt$, say $R_1,\cdots, R_d$, where $R_i=\tau^{-(i-1)}R_1$ for $1<i\leq d$.  Each indecomposable module of $\mt$ is an iterated extension of $R_1,\cdots, R_d$.
Moreover, for each $1\leq i\leq d$ and $l\in \N$, there is a unique indecomposable module $M_{(i,l)}$ in $\mt$ with socle $R_i$ and quasi-length $l$. These modules exhaust the indecomposable modules in $\mt$.
We may use the coordinate system to denote the indecomposable objects in $\mt$. Namely, denote by $(a,b)$ the unique indecomposable object in $\mt$ with socle $R_a$ and quasi-length $b$.  When the first argument does not belong to $[1,d]$, we will implicitly assume identification modulo $d$ (see { Figure 1.} for an example of a tube of rank $3$).
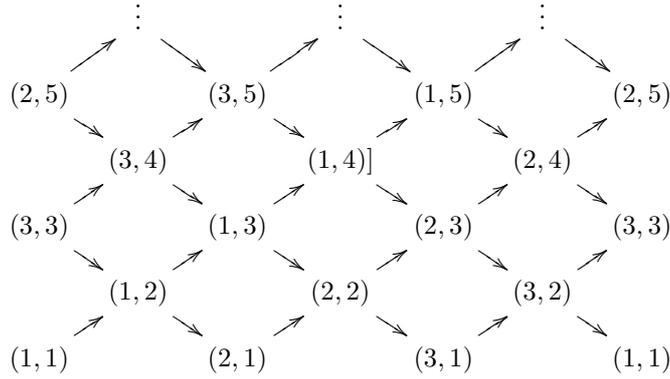
\begin{figure}
\[\xymatrix@R=0.25cm@C=0.25cm{&~~\vdots~~\ar[dr]&&~~\vdots~~\ar[dr]&&~~\vdots~~\ar[dr]&\\
(2,5)\ar[ur]\ar[dr]&&(3,5)\ar[ur]\ar[dr]&&(1,5)\ar[ur]\ar[dr]&&(2,5)\\
&(3,4)\ar[dr]\ar[ur]&&(1,4)\ar[dr]\ar[ur]]&&(2,4)\ar[ur]\ar[dr]&\\
(3,3)\ar[dr]\ar[ur]&&(1,3)\ar[dr]\ar[ur]&&(2,3)\ar[dr]\ar[ur]&&(3,3)\\
&(1,2)\ar[dr]\ar[ur]&&(2,2)\ar[dr]\ar[ur]&&(3,2)\ar[ur]\ar[dr]\\
(1,1)\ar[ur]&&(2,1)\ar[ur]&&(3,1)\ar[ur]&&(1,1)}
\]
\caption{The AR quiver of a tube $\mt$ of rank $3$. The quasi-simple modules are $(1,1), (2,1)$ and $(3,1)$.}
\end{figure}

For an indecomposable object $(a,b)\in \mt$, the infinite sequence of irreducible maps
\[\mathrm{R}_{(a,b)}=(a,b)\to (a, b+1)\to \cdots \to (a,b+j)\to \cdots
\]
is called a {\it ray} starting at $(a,b)$ and the infinite sequence of irreducible maps
\[\mathrm{C}_{(a,b)}=\cdots\to (a-j,b+j)\to \cdots\to (a-1,b+1)\to (a,b)
\]
is called a {\it coray} ending at $(a,b)$. We also write
\[\mathrm{R}_{(a,b)}=\{(a,b+i)~|~\forall i\geq 0\}~\text{and}~\mathrm{C}_{(a,b)}=\{(a-j,b+j)~|~\forall j\geq 0\}.
\]
The {\it wing} $\cw_{(a,b)}$ determined by $(a,b)$ is defined to be the set
\[\cw_{(a,b)}=\{(a', b')~|~a'\geq a~\text{and}~a'+b'\leq a+b\}.
\]
Recall that a $kQ$-module $M$ is {\it rigid} if $\Ext^1_{kQ}(M,M)=0$.
We end up this subsection with the following facts of $\mod kQ$ ({\it cf.}~\cite{Ri84,SS} for instance).
\begin{proposition}~\label{p:property-of-tame-quiver}
Let $Q$ be a connected extended Dynkin quiver.
\begin{itemize}
\item[(1)] The indecomposable rigid $kQ$-modules are precisely the preprojectives, the preinjectives and the indecomposable regular $kQ$-modules $M$ with quasi-length $q.l.(M)<t$ in a rank $t$ tube, where $t>1$;
\item[(2)] Let $M$ be a preprojective $kQ$-module,  $R$ a regular $kQ$-module and $L$ a preinjective $kQ$-module, then $\Hom_{kQ}(L,R)=\Hom_{kQ}(L,M)=\Hom_{kQ}(R,M)=0$. Moreover, there are no nonzero morphisms between different tubes;
\item[(3)] For any tilting $kQ$-module $T$, the number of non-isomorphic indecomposable regular direct summands of $T$ is at most $|kQ|-2$;
\item[(4)] For a tilting $kQ$-module $T$, if $T$ has a regular direct summand $M$, then T has an indecomposable regular direct summand $N$ which lies in the same tube as
 $M$ with $q.l.(N)=1$.
\end{itemize}
\end{proposition}

\subsection{Cluster categories}~\label{ss:cluster-category} We follow~\cite{BMRRT}.
Let $Q$ be a finite acyclic quiver (may be non-connected) and $kQ$ the path algebra of $Q$ over $k$. Denote by $\der^b(\mod kQ)$ the bounded derived category of $\mod kQ$ with suspension functor $\Sigma$. Let $\tau$ be the Auslander-Reiten translation of $\der^b(\mod kQ)$, which is  an autoequivalence of $\der^b(\mod kQ)$. The {\it cluster category $\mc_Q$} associated to $Q$ is the orbit category $\der^b(\mod kQ)/\tau^{-1}\circ \Sigma$ which was introduced in~\cite{BMRRT} in order to give a categorical framework of Fomin-Zelevinsky's cluster algebras ({\it cf.}~\cite{CCS06} for $A_n$ case). Namely, the cluster category $\mc_Q$ has the same objects as $\der^b(\mod kQ)$ and the morphism spaces
\[\Hom_{\mc_Q}(X,Y)=\bigoplus_{i\in \Z}\Hom_{\der^b(\mod kQ)}(X,(\tau^{-1}\circ \Sigma)^iY)~\text{for ~$X,Y\in \der^b(\mod kQ)$}.
\]
It has been proved by Keller~\cite{K} that $\mc_Q$ admits a canonical triangle structure such that the projection $\pi_Q:\der^b(\mod kQ)\to \mc_Q=\der^b(\mod kQ)/\tau^{-1}\circ \Sigma$ is a triangle functor.
Moreover, $\mc_Q$ is a Calabi-Yau triangulated category of dimension $2$. Note that we have $\tau=\Sigma$ in $\mc_Q$.

For any finite acyclic quiver $Q'$ such that $kQ'$ is derived equivalent to $kQ$, we clearly have an equivalence of triangulated categories $\mc_Q\cong \mc_{Q'}$. Via this equivalence, we may regard $kQ'$-modules as objects of $\mc_Q$ and $\ind (\mod kQ'\vee~\Sigma (kQ'))$ forms a representative set of indecomposable objects of $\mc_Q$, where the set $\ind (\mod kQ'\vee~\Sigma (kQ'))$ consists of the indecomposable $kQ'$-modules together with the objects $\Sigma P$, where $P$ runs over indecomposable projective $kQ'$-modules.

The following result is useful to compute morphism spaces in $\mc_Q$, which will be used frequently ({\it cf.}~\cite{BMRRT}).
\begin{lemma}~\label{l:compute-morphism}
For any $kQ$-modules $X, Y$, we have
\[\Hom_{\mc_Q}(X,Y)=\Hom_{\der^b(\mod kQ)}(X,Y)\oplus \Hom_{\der^b(\mod kQ)}(X, \tau^{-1}\circ\Sigma Y).
\]
Moreover, if $X$ is a projective $kQ$-module or $Y$ is an injective $kQ$-module, then
\[\Hom_{\mc_Q}(X,Y)=\Hom_{\der^b(\mod kQ)}(X,Y).
\]
If both $X$ and $Y$ are regular $kQ$-modules, then
\[\Hom_{\mc_Q}(X,Y)=\Hom_{\der^b(\mod kQ)}(X,Y)\oplus D\Hom_{\der^b(\mod kQ)}(Y, \tau^2 X).
\]
\end{lemma}
Let $\ind^{\operatorname{rig}}\mc_Q$ be the set of isoclasses of indecomposable rigid objects in $\mc_Q$.
An easy consequence of Lemma~\ref{l:compute-morphism} is that the set $\ind^{\operatorname{rig}}\mc_Q$ consists of indecomposable rigid $kQ$-modules together with the objects $\Sigma P$, where $P$ runs over indecomposable projective $kQ$-modules.

The structure of the AR quiver of $\der^b(\mod kQ)$ is well-understood via the AR quiver of $\mod kQ$ ({\it cf.}~\cite{Ha}). By the fact that the projection $\pi_Q:\der^b(\mod kQ)\to \mc_Q$ preserves the Auslander-Reiten triangles, we deduce the AR quiver of $\mc_Q$ from the AR quiver of $\der^b(\mod kQ)$. Namely, assume moreover that $Q$ is connected, 
\begin{itemize}
\item[$\circ$] if $Q$ is a Dynkin quiver, then the AR quiver of $\mc_Q$ admits a unique connected component with finitely many indecomposable objects;
\item[$\circ$] if $Q$ is an extended Dynkin quiver, then the AR quiver of $\mc_Q$ consists of a connected AR component with shape $\mathbb{Z}Q$ (the translation quiver of $Q$), infinitely many homogeneous tubes  and finitely many non-homogeneous tubes;
\item[$\circ$] if $Q$ is of wild type, then the AR quiver of $\mc_Q$ consists of a connected AR component with shape $\mathbb{Z}Q$ and infinitely many connected AR components with shape $\mathbb{Z}A_{\infty}$.
\end{itemize}

Using the AR quiver of a cluster category, we introduce the following intrinsic
definition of regular and transjective objects in cluster categories.
\begin{definition}
Let $\mc$ be a cluster category. For any indecomposable object $M\in \mc$, denote by  $\Gamma_M$ the connected AR component of $\mc$ containing $M$. An indecomposable object $M\in \mc$ is called {\it regular}, if there is an indecomposable object $N$ lying in the AR component $\Gamma_M$ such that $\dim_k\Hom_\mc(N,N)\geq 2$. Otherwise, $M$ is called {\it transjective}.
\end{definition}
We also call an object $M\in \mc$ {\it transjective} (resp. {\it regular}), if $M$ has no indecomposable regular (resp. {\it transjective}) direct summand.
It is easy to see that our definition coincides with the one in~\cite{BM10}. Indeed, if $\mc=\mc_Q$ for a finite acyclic quiver $Q$, then the indecomposable regular objects of $\mc_Q$ are precisely the indecomposable regular $kQ$-modules.

A cluster category $\mc_Q$ is of {\it finite type} if there are only finitely many indecomposable objects. In other words, $Q$ is a disjoint union of Dynkin quivers. Using Auslander-Reiten theory of a cluster category,  we obtain the following simple observation.
\begin{lemma}~\label{l:AR-shape}
Let $\mc$ be a cluster category.
\begin{itemize}
\item[(1)] $\mc$ is of finite type if and only if $\mc$ has only finitely many indecomposable rigid objects;
\item[(2)]  Assume that $\mc$ admits a connected AR component $\Gamma_{\circ}$ with finitely many indecomposable objects. Let $\mc_\circ$ be the full subcategory of $\mc$ consisting of objects which are finite direct sum of indecomposable objects lying in $\Gamma_\circ$. Then $\mc_\circ$ is a cluster category of finite type. Moreover, there is triangulated subcategory $\mc_{\bullet}$ of $\mc$ such that $\mc$ is the direct sum of $\mc_\circ$ with $\mc_\bullet$.
\end{itemize}
\end{lemma}

\subsection{Cluster-tilted algebras}
 The cluster-tilting objects of $\mc_Q$ have close relationship with the tilting theory of $kQ$. Among others, the following result has been proved in~\cite{BMRRT}.
\begin{lemma}~\label{l:cluster-vs-tilting}
\begin{itemize}
\item[(a)] An object $T\in \mc_Q$ is a cluster-tilting object if and only if there is a finite-dimensional hereditary algebra $kQ'$ which is derived equivalent to $kQ$ such that $T$ is identified to a tilting $kQ'$-module via the equivalence $\mc_Q\cong \mc_{Q'}$;
\item[(b)] Any two basic cluster-tilting objects of $\mc_Q$ can be obtained from each other by a series of mutations.
\end{itemize}
\end{lemma}
As an immediate consequence, we know that all the basic cluster-tilting objects of $\mc_Q$ have exactly $|kQ|$ indecomposable direct summands.

Let $T$ be a basic cluster-tilting object in ~$\mc_Q$. The endomorphism algebra $\End_{\mc_Q}(T)^{\op}$ is called a {\it cluster-tilted algebra}~\cite{BMR07}.
By definition, it is clear that the hereditary algebra $kQ$ is a cluster-tilted algebra. The following results have been established in~\cite{BMR07,BMR08}.
\begin{theorem}~\label{t:cluster-tilted-algebras}
Let $T$ be a basic cluster-tilting object in~$\mc_Q$ and $\Gamma=\End_{\mc_Q}(T)^{\op}$ the associated cluster-tilted algebra.
\begin{itemize}
\item[(a)] The functor $\Hom_{\mc_Q}(T,-):\mc_Q\to \mod \Gamma$ induces an equivalence \[\mc_Q/\add \Sigma T\xrightarrow{\sim} \mod \Gamma,\] where $\mc_Q/\add \Sigma T$ is the additive quotient of $\mc_Q$ by the subcategory $\add \Sigma T$;
\item[(b)] Let $e$ be an idempotent of $\Gamma$. The quotient algebra $\Gamma/\Gamma e\Gamma$ is again a cluster-tilted algebra;
\item[(c)] The quiver $Q_\Gamma$ of $\Gamma$ has no loops nor $2$-cycles.
\end{itemize}
\end{theorem}

Combining Theorem~\ref{t:cluster-structure} with Theorem~\ref{t:cluster-tilted-algebras}~$(c)$, we conclude that the cluster category $\mc_Q$ admits a cluster structure.
In the following, for each object $M\in \mc_Q$ without indecomposable direct summands in $\add\Sigma T$, we will denote its image in $\mod \Gamma$ by $M_\Gamma$ and each $\Gamma$-module is given of the form $M_\Gamma$ for some object $M\in \mc_Q$ without indecomposable direct summands in $\add\Sigma T$  by Theorem~\ref{t:cluster-tilted-algebras}~$(a)$.

Let $\Lambda$ be a finite-dimensional algebra over $k$ and $\mod \Lambda$ the category of finitely generated left $\Lambda$-modules. Denote by $\tau$ the Auslander-Reiten translation of $\mod \Lambda$. Recall that a module $M\in \mod \Lambda$ is {\it $\tau$-rigid} if $\Hom_\Lambda(M,\tau M)=0$. For a cluster-tilted algebra $\Gamma$, the following bijection has been established in~\cite{AIR}.
\begin{lemma}~\label{l:bijection-tau-rigid-and-rigid}
Let $T$ be a basic cluster-tilting object of $\mc_Q$ and $\Gamma=\End_{\mc_Q}(T)^{\op}$ the associated cluster-tilted algebra. The functor $\Hom_{\mc_Q}(T,-)$ induces a bijection between the isoclasses of indecomposable rigid objects of $\mc_Q\backslash\add\Sigma T$ and the isoclasses of indecomposable $\tau$-rigid $\Gamma$-modules.
\end{lemma}

As a consequence of  Theorem~\ref{t:cluster-tilted-algebras}~$(a)$ and Lemma~\ref{l:bijection-tau-rigid-and-rigid}, we have
\begin{corollary}
Let $\Gamma=\End_{\mc_Q}(T)^{\op}$ and $\Gamma'=\End_{\mc_Q}(T')^{\op}$ be two cluster-tilted algebras.  Let $M\in \mc_Q\backslash \add \Sigma T$ be an indecomposable object such that $M_{\Gamma}$ is $\tau$-rigid.
If $M\not\in \add\Sigma T'$, then $M_{\Gamma'}$ is also an indecomposable $\tau$-rigid $\Gamma'$-module.
\end{corollary}

\subsection{Cluster categories of tame type}
 The cluster category $\mc_Q$ is
 of {\it tame type} if $Q$ is a connected extended Dynkin quiver. In this case, a cluster-tilted algebra of $\mc_Q$ is also called a {\it cluster-tilted algebra of tame type}.

Now assume that $Q$ is a connected extended Dynkin quiver and $\mc_Q$ the associated cluster category of tame type. The projection functor $\pi_Q$ yields a bijection between the indecomposable regular $kQ$-modules and the indecomposable regular objects of $\mc_Q$. Moreover, each indecomposable regular object of $\mc_Q$ lies in a tube of $\mc_Q$.
 For an indecomposable regular object $M\in \mc_Q$, we denote by $\mt_M$ the tube where $M$ lies in. We may define its quasi-length $q.l.(M)$ to be the quasi-length of the corresponding regular $kQ$-module. The ray $\mathrm{R}_{M}$, the coray $\mathrm{C}_M$ and the wing $\cw_M$ determined by $M$ in $\mt_M$ can be defined similarly. For a fixed tube $\mt$ of $\mc_Q$, we still 
 use the coordinate system to denote the indecomposable objects in $\mt$.
 If $\mt$ is a tube of rank $d$ in $\mc_Q$, then the functor $\Sigma^d$ restricts to the identity on $\mt$ and $\Sigma (a, b)=\tau (a,b)=(a-1, b)$ for any $(a, b)\in \mt$. Notice that the tube $\mt$ of $\mc_Q$ is no longer standard. Nevertheless, the morphism space between any two indecomposable regular objects can be computed easily using Lemma~\ref{l:compute-morphism}.

Let $\mt$ be a tube of $\mc_Q$. For any indecomposable objects $M,N\in \mt$, we have 
\[\Hom_{\mc_Q}(M,N)=\Hom_{\der^b(\mod kQ)}(M,N)\oplus D\Hom_{\der^b(\mod kQ)}(M,\tau^{-1}\Sigma N).
\]
Following~\cite{BMV}, morphisms in $\Hom_{\der^b(\mod kQ)}(M,N)$ are called {\it $\mathcal{M}$-maps} from $M$ to $N$ and morphisms in $\Hom_{\der^b(\mod kQ)}(M,\tau^{-1}\circ \Sigma N)$ are called {\it $\der$-maps} from X to Y . Each morphism from $M$ to $N$ in $\mc_Q$ can be written as the sum of an $\mathcal{M}$-map with a $\der$-map. The composition of two $\mathcal{M}$-maps is also a $\mathcal{M}$-map, the composition of an $\mathcal{M}$-map with a $\der$-map is a $\der$-map,  the composition of two $\der$-maps is zero, and no $\mathcal{M}$-map can factor through a $\der$-map.

The following lemma summarizes certain facts about cluster categories of tame type.
 \begin{lemma}~\label{l:property-tame-clustercategory}
 Let $\mc_Q$ be a cluster category of tame type. 
 \begin{itemize}
 \item[(1)] An indecomposable object $M\in \mc_Q$ is rigid if and only if $M$ is transjective or $q.l.(M)\leq d-1$ if $M$ lies in a tube of rank $d>1$;
 \item[(2)] There are no nonzero morphisms between different tubes in $\mc_Q$;
 \item[(3)] Let $T$ be a basic cluster-tilting object of $\mc_Q$. If $T$ has an indecomposable regular direct summand $M$, then $T$ has an indecomposable regular direct summand $N$ which lies in the tube $\ct_M$ with $q.l.(N)=1$;
 \item[(4)] Let $\mt$ be a tube of $\mc_Q$ with rank $d>1$.  Let $M$ be an indecomposable object of $\mt$ with $q.l.(M)=1$ and $N$ an indecomposable rigid object lying in the ray $\mathrm{R}_M$ or the coray $\mathrm{C}_{\tau^2M}$. Then $\dim_k\Hom_{\mc_Q}(M,N)\leq 2$ and the equality holds if and only if  $q.l.(N)=d-1$.
 \end{itemize}
 \end{lemma}
 \begin{proof}
 Part $(1)$ is a direct consequence of Proposition~\ref{p:property-of-tame-quiver} (1) and Lemma~\ref{l:compute-morphism}. Part $(2)$ follows from Proposition~\ref{p:property-of-tame-quiver} (3) and Lemma~\ref{l:compute-morphism}.
 To deduce part $(3)$, we use Proposition~\ref{p:property-of-tame-quiver} (4) and Lemma~\ref{l:cluster-vs-tilting} (a). 
 
 For $(4)$, without loss of generality, we may assume that $M=(1,1)$. Note that $N$ is rigid which implies that $q.l.(N)\leq d-1$ by (1).  We have $N=(1,i)$ or $N=(i,d-i)$ for some $1\leq i\leq d-1$.  By Lemma~\ref{l:compute-morphism}, $\Hom_{\mc_Q}(M,N)=\Hom_{kQ}(M,N)\oplus D\Hom_{kQ}(N,\tau^2M)$. Since each tube of $\mod kQ$ is standard, we  have
 \[\Hom_{kQ}(M,N)=k~\text{if and only if}~N=(1,i)
 \]
 \[\Hom_{kQ}(N,\tau^2M)=k~\text{if and only if}~N=(i,d-i)
 \]
 for $1\leq i\leq d-1$. Consequently, $\dim_k\Hom_{\mc_Q}(M,N)\leq 2$ and the equality holds if and only if $N=(1,d-1)$.
 \end{proof}

 Recall that $\ind^{\operatorname{rig}}\mc_Q$ is the set of isoclasses of indecomposable rigid objects of $\mc_Q$.
 \begin{lemma}~\label{l:hom-vanish-set}
  Let $\mc_Q$ be a cluster category of tame type and $M$  an indecomposable transjective object. For any nonnegative integer $l$, both the sets
  \[\{N\in \ind^{\operatorname{rig}} \mc_Q~|~\dim_k\Hom_{\mc_Q}(M,N)\leq l\}~\text{and}~\{L\in \ind^{\operatorname{rig}} \mc_Q~|~\dim_k\Hom_{\mc_Q}(L,M)\leq l\}
  \]
  are finite.
 \end{lemma}
 \begin{proof}
 Since $M$ is transjective, we may choose a tame hereditary algebra $H=kQ'$, which is derived equivalent to $kQ$, such that $M$ is identified to an indecomposable projective $H$-module. In this case, each indecomposable object of $\mc_Q$ is identified to an indecomposable $H$-module or $\Sigma P$ for some indecomposable projective $H$-module $P$.
 
 Let $\mathfrak{g}_{Q'}$ be the affine Kac-Moody algebra associated to $Q'$ and $\dot{\mathfrak{g}}$ be the associated simple Lie algebra of $\mathfrak{g}_{Q'}$.
 Denote by $\delta$ the minimal positive imaginary root of $\mathfrak{g}_{Q'}$. It is well-known that each positive real root of $\mathfrak{g}_{Q'}$ can be written as the sum of a  positive root of $\dot{\mathfrak{g}}$ with $m\delta$ for some positive integer $m$ ({\it cf.}~\cite{Kac90}). For a  vertex $i$ of $Q'$ and an integer $l$, denote by $\mathcal{S}_{i,l}$  the set consisting of positive roots of $\mathfrak{g}_{Q'}$ whose $i$-th components are less than $l$.  We clearly know that $\mathcal{S}_{i,l}$ is a finite set.
 
According to~\cite{Kac80}, the dimension vector of an indecomposable rigid $H$-module is a positive root of $\mathfrak{g}_{Q'}$. On the other hand,  different rigid $H$-modules have different dimension vectors~\cite{CB}.
 For a fixed indecomposable projective $H$-module $P_i$ and a nonnegative integer $l$, one deduces that  
 \[|\{Z\in \ind^{\operatorname{rig}} H~|~\dim_k\Hom_{H}(P_i,Z)\leq l\}|\leq|\mathcal{S}_{i,l}|<\infty,
 \]  where $\ind^{\operatorname{rig}} H$ is the set of isoclasses of indecomposable rigid $H$-modules. 
Now for any indecomposable $H$-module $N$, we have
$\Hom_{\mc_Q}(M,N)\cong \Hom_{H}(M,N)$ by Lemma~\ref{l:compute-morphism}.
 Consequently, the set $\{N\in \ind^{\operatorname{rig}} \mc_Q~|~\dim_k\Hom_{\mc_Q}(M,N)\leq l\}$ is finite.

 For the second set, one choose a tame hereditary algebra $H'$ such that $M$ is  identified to an indecomposable injective $H'$-modules.
 \end{proof}

\section{Subfactors of cluster categories}~\label{S:subfactor}
\subsection{Subfactors of triangulated categories}
Subfactors of triangulated categories were introduced in~\cite{IY}. In this subsection, we recall the basic constructions and results of subfactors for $2$-Calabi-Yau categories and we refer to~\cite{IY} for the general situation.

Let $\mc$ be a $2$-Calabi-Yau triangulated category with suspension functor $\Sigma$.
For a given functorially finite rigid subcategory $\mz$ of $\mc$, set
\[\mb:=~^\perp(\Sigma\mz):=\{X\in\mc~|~\Hom_{\mc}(X,\Sigma Z)=0~\text{for any ~$Z\in\mz$}\}.
\]
It is not hard to see that $\mb$ is functorially finite and extension closed. Moreover, $\mz\subseteq \mb$ and one may form the additive quotient category $\mb/\mz$. Surprisingly, it has been shown in~\cite{IY} that the quotient category $\mb/\mz$ inherits a triangle structure from $\mc$. Moreover, the following results has been obtained in~\cite{IY} ({\it cf.} also~\cite{BIRS}).
\begin{theorem}~\label{t:subfactor}
The category $\mb/\mz$ is also a $2$-Calabi-Yau triangulated category. Moreover,
 \begin{itemize}
 \item[(1)] an object $M\in \mb$ is rigid in $\mc$ if and only if $M$ is rigid in $\mb/\mz$;
 \item[(2)] there is a one-to-one correspondence between cluster-tilting subcategories of $\mc$ containing $\mz$ and cluster-tilting subcategories of $\mb/\mz$.
 \end{itemize}
\end{theorem}

For the later use, let us briefly  recall the triangle structure of $\mb/\mz$. The suspension functor $\mathbb{G}:\mb/\mz\to \mb/\mz$ is defined as follows.
For each $X\in\mb$, consider the left minimal $\mz$-approximation $X\xrightarrow{\alpha_X}Z_X$ of $X$ and form a triangle in $\mc$
\[X\xrightarrow{\alpha_X}Z_X\xrightarrow{\beta_X}Y\to \Sigma X.
\]
One can check that $Y\in \mb$ and set $\mathbb{G}(X):=Y$. For a morphism $f\in \Hom_{\mc}(X,X')$ with $X,X'\in \mb$, there exists $g$ and $h$ which make the following diagram commutative
\[\xymatrix{X\ar[d]_f\ar[r]^{\alpha_X}&Z_X\ar[d]_g\ar[r]^{\beta_X}&Y\ar[d]_h\ar[r]&\Sigma X\ar[d]_{\Sigma f}\\
X'\ar[r]^{\alpha_{X'}}&Z_{X'}\ar[r]^{\beta_{X'}}&Y'\ar[r]&\Sigma X'}
\]
and we define $\mathbb{G} (\overline{f}):=\overline{h}$, where for a morphism $f\in \Hom_{\mc}(X,X')$ with $X,X'\in \mb$, we denote by $\overline{f}$ the image of $f$ in the quotient category $\mb/\mz$. As shown in~\cite{IY}, this gives a well-defined autoequivalence.

Now we turn to describe the standard triangles in $\mb/\mz$.
Let $X\xrightarrow{a}Y\xrightarrow{b}Z\xrightarrow{c}\Sigma X$ be a triangle in $\mc$ with $X,Y,Z\in \mb$. By definition of $\mb$, there is a commutative diagram of triangles
\[\xymatrix{X\ar@{=}[d]\ar[r]^a&Y\ar[d]\ar[r]^b&Z\ar[d]^d\ar[r]^c&\Sigma X\ar@{=}[d]\\
X\ar[r]^{\alpha_X}&Z_X\ar[r]^{\beta_X}&\mathbb{G} X\ar[r]&\Sigma X.}
\]
Note that the morphism $d$ is not unique, but one can  prove that different choices of the morphism $d$ yield the same image in $\mb/\mz$.
The standard triangle of $\mb/\mz$ is defined to be the  complex $X\xrightarrow{\overline{a}}Y\xrightarrow{\overline{b}}Z\xrightarrow{\overline{d}}\mathbb{G} X$ in $\mb/\mz$.

Recall that a triangulated category $\der$ is called {\it algebraic} if there is a Frobenius category $\mf$ such that $\der$ is the stable category $\underline{\mf}$ of $\mf$. We refer to~\cite{Ha} for the definition and the standard triangles of $\underline{\mf}$. The subfactors of Frobenius $2$-Calabi-Yau categories were investigated in~\cite{BIRS}.

Combining the results of~\cite{IY,BIRS}, we have the following observation.
\begin{lemma}~\label{l:algebraic-subfactor}
Let $\mc$ be an algebraic $2$-Calabi-Yau triangulated category and $\mz$ a functorially finite rigid subcategory of $\mc$. Set $\mb:=~^\perp(\Sigma \mz)$. Then the subfactor triangulated category $\mb/\mz$ is also algebraic.
\end{lemma}
\begin{proof}
Let $\mf$ be a Frobenius category such that $\underline{\mf}=\mc$. Denote by $\mathcal{P}$ the full subcategory of $\mf$ consisting of projective-injective objects of $\mf$. Let $\hat{\mz}$ be the preimage of $\mz$ under the projection $\pi:\mf\to \underline{\mf}=\mc$ and define
\[\hat{\mb}:=\{X\in \mf~|~\Ext^1_{\mf}(X, Z)=0~\text{for $Z\in \hat{\mz}$}\}.
\]
Clearly, we have $\hat{\mb}/\mathcal{P}=\mb$. Recall that $\mz$ is functorially finite rigid in $\mc$ implies that $\mb$ is also functorially finite in $\mc$, from which one can deduce that $\hat{\mb}$ is functorially finite in $\mf$. Moreover, by definition of $\hat{\mb}$, it is clear that $\hat{\mb}$ is extension closed. Now by Theorem II $2.6$ of~\cite{BIRS}, we know that $\hat{\mb}$ is a Frobenius category and $\hat{\mz}$ is the subcategory of  projective-injective objects. Moreover, the stable category $\underline{\hat{\mb}}=\hat{\mb}/\hat{\mz}$ is the same as the subfactor category $\mb/\mz$ as additive category. Now it is straightforward to check that  the standard triangles given by the Frobenius structure of $\hat{\mb}$ coincide with the standard triangles given by the subfactor $\mb/\mz$.
\end{proof}

\subsection{Subfactors of cluster categories}
In this subsection, we apply Iyama-Yoshino's construction of subfactors to cluster categories.  
Namely, we have the following main result of this subsection.
\begin{theorem}~\label{t:subfactor-cluster-category}
Let $\mc$ be a cluster category and $Z$ a rigid object. The $2$-Calabi-Yau triangulated category $^\perp(\add \Sigma Z)/\add Z$ is a cluster category. Moreover, there is a one-to-one correspondence between cluster-tilting objects of $\mc$ containing $Z$ and cluster-tilting objects of $^\perp(\add \Sigma Z)/\add Z$.
\end{theorem}
\begin{proof}
According to Theorem~\ref{t:subfactor}, it remains to prove that $^\perp(\add \Sigma Z)/\add Z$ is a cluster category. It follows from Keller's construction~\cite{K} that the cluster category $\mc$ is algebraic. Consequently, the subfactor $^\perp(\add \Sigma Z)/\add Z$ is also algebraic by Lemma~\ref{l:algebraic-subfactor}.
By Keller-Reiten's recognition theorem of cluster category~\cite{KR}, it suffices to show that there is a basic cluster-tilting object $M\in~ ^\perp(\add \Sigma Z)/\add Z$ such that its quiver $Q_M$ is acyclic.

Let $T$ be a cluster-tilting object in $^\perp(\add \Sigma Z)/\add Z$, then $T\oplus Z$ is a cluster-tilting object in $\mc$. Denote by $\Gamma=\End_{\mc}(T\oplus Z)^{op}$ the cluster-tilted algebra of $T\oplus Z$ and $e_Z$ the idempotent associated to $Z$. We clearly have $\End_{^\perp(\add \Sigma Z)/\add Z}(T)^{op}\cong \Gamma/\Gamma e_Z\Gamma$, which is a cluster-tilted algebra by Theorem~\ref{t:cluster-tilted-algebras}. Moreover, the quiver of $\End_{^\perp(\add \Sigma Z)/\add Z}(T)^{op}$ has no loops nor $2$-cycles. Consequently, the $2$-Calabi-Yau category $^\perp(\add \Sigma Z)/\add Z$ admits a cluster structure by Theorem~\ref{t:cluster-structure}. Let $Q$ be the quiver of $\End_{^\perp(\add \Sigma Z)/\add Z}(T)^{op}$. The quiver $Q$ is the quiver of a cluster-tilted algebra implies that there is a sequence $i_1,\cdots, i_t$ of vertices of $Q$ such that the quiver corresponding to the skew-symmetric matrix $\mu_{i_t}\circ\cdots \mu_{i_1}(B(Q))$ is acyclic. By definition of cluster structure, applying the  corresponding mutations to the cluster-tilting object $T$ in $^\perp(\add \Sigma Z)/\add Z$, we obtain a cluster-tilting object $M\in~ ^\perp(\add \Sigma Z)/\add Z$ such that the quiver $Q_M$ is acyclic. This finishes the proof.

\end{proof}

\subsection{Subfactors of cluster categories of tame type} In this subsection, we study the subfactor of cluster categories of tame type with respect to an indecomposable rigid object.
 The following easy result on determining AR-triangles will be used frequently.
\begin{lemma}~\label{l:AR-triangles}
Let $\mc$ be a $2$-Calabi-Yau triangulated category with suspension functor $\Sigma$.
\begin{itemize}
\item[(1)] Let $X,Z, Y_1,\cdots, Y_t$ be non-isomorphic indecomposable objects of $\mc$. Assume that 
\[X\xrightarrow{f}Y_1\oplus\cdots\oplus Y_t\xrightarrow{g}Z\to \Sigma X\]
is a triangle of $\mc$ such that each component of $f$ and $g$ is irreducible, then the above triangle is an AR-triangle of $\mc$. In particular, $\Sigma Z\cong X$;
\item[(2)] Let $M\in \mc$ be a rigid object and $\mc':=~^\perp(\Sigma M)/\add M$ the subfactor determined by $M$. Let $X\xrightarrow{f} Y\xrightarrow{g} Z\to \Sigma X$ be an AR-triangle of $\mc$ such that $X,Y,Z\in ~^\perp(\Sigma M)$, then $X\xrightarrow{\overline{f}}Y\xrightarrow{\overline{g}}Z\to \mathbb{G}X$ is an AR-triangle of $\mc'$, where $\mathbb{G}$ is the suspension functor of $\mc'$.
\end{itemize}
\end{lemma}

In the following, we assume that $Q$ is a connected extended Dynkin quiver and $\mc_Q$ is the cluster category associated to $Q$.  Let $M$ be an indecomposable rigid object of $\mc_Q$ and  $\mc:=~^\perp(\Sigma M)/\add M$ the subfactor of $\mc_Q$ determined by $M$.
By Theorem~\ref{t:subfactor-cluster-category}, we know that $\mc$ is a cluster category.
Assume moreover that $M$ is regular in $\mc_Q$. We are going to determine the connected AR components of $\mc$ which are induced from tubes of $\mc_Q$.
Recall that $\mc$ is a $2$-Calabi-Yau triangulated category by Theorem~\ref{t:subfactor}, which implies that its Auslander-Reiten translation coincides with the suspension functor $\mathbb{G}$ of $\mc$.

  For a tube $\mt$ of $\mc_Q$ such that $M\not\in \mt$, we have $\mt\subset ~^\perp(\Sigma M)$ by Lemma~\ref{l:property-tame-clustercategory}~(2). It follows that the tube $\mt$ remains a tube with the same rank  in the subfactor~$\mc$ by Lemma~\ref{l:AR-triangles}~(2).

Now we turn to the tube $\mt_M$ containing $M$. Without loss of generality, we may assume that the rank of $\mt_M$ is $d$ and $M=(1,t)$, where $t<d$.
Let $\mt_{M,\mc}$ be a representative set of indecomposable objects of $\mt_M$ which  do not lie on the corays ending at $(d,1), (1,1),(2,1),\cdots, (t-1,1)$ or the rays starting at $(2,1), (3,1),\cdots, (t+1,1)$.
By Lemma~\ref{l:compute-morphism}, we know that
\[~^\perp(\Sigma M)\cap \mt_M=\mt_{M, \mc}\cup \cw_{M}.
\]
See Figure 2. for an illustration of $\mt_{M,\mc}\cup \cw_M$.

\begin{figure}
\begin{center}
\begin{tikzpicture}[
    scale=4,axis/.style={ very thick, ->, >=stealth'},
    arrow/.style={thick, ->, >=stealth'},
    important line/.style={thick},
    dashed line/.style={dashed, thin},
    pile/.style={thick, ->, >=stealth', shorten <=2pt, shorten
    >=2pt},
    every node/.style={color=black}
    ]
  
    \draw[dashed line](-1,0) --(1,0)[right];
    \draw[dashed line](-1,.4)--(1,.4)[right];
    \draw[color=red,important line](-.6,0)--(-.2,.4)[right]node[above]{\tiny{$(d,t)$}};
    \node at(-.6,0){\tiny{$\circ$}};
    \node at(-.2,.4){\tiny{$\circ$}};
    \node at (-.6,-.1){\tiny{$(d,1)$}};
    \node at(0,.4){\tiny{$\circ$}};
    \node at(0,.45){\tiny{M}};
    \node at(.4,0){\tiny{$\circ$}};
    \node at(.4,-.1){\tiny{$(t,1)$}};
    \node at(-.4,-.1){\tiny{$(1,1)$}};
    \node at(-.4,0){\tiny{$\circ$}};
    \node at(0,.2){\tiny{$\cw_M$}};
    \draw[important line](0,.4)--(-.4,0)[left];
    \draw[important line](0,.4)--(.4,0)[left];
    \draw[color=red,important line](.6,0)--(.2,.4)[left]node[above]{\tiny{$(2,t)$}};
    \node at(.2,.4){\tiny{$\circ$}};
    \node at(.6,0){\tiny{$\circ$}};
    \node at(.65, -.1){\tiny{$(t+1,1)$}};
    \node at(-.6,.45){\tiny{$L_M$}};
    \node at(.6,0.45){\tiny{$R_M$}};
    \draw[color=red,important line](-.6,0)--(-1,.4)[right];
    \draw[color=red,important line](-.2,.4)--(-.8,1)[right];
    \draw[color=red,important line](.6,0)--(1,.4)[right];
    \draw[color=red,important line](0.2, 0.4)--(.8,1)[right];
\end{tikzpicture}
\end{center}
\caption{$\mt_{M,\mc}\cup \cw_M$ consists of objects which do not lie in $L_M$ and $R_M$.}
\end{figure}
\begin{lemma}~\label{l:AR-tube-finite}
Indecomposable objects in $\cw_M\backslash \add M$ form a connected AR component $\Gamma_{\circ}$ of $\mc$. Moreover, $\Gamma_{\circ}$ is the AR quiver of  a cluster category of a Dynkin quiver of type $A_{t-1}$.
\end{lemma}
\begin{proof}
For each $1\leq i\leq t-1$, we have the following triangle in ~$\mc_Q$
\[(1,i)\to (1,t)\to (i+1,t-i)\to \Sigma (1,i).
\]
By definition of the suspension functor ~$\mathbb{G}$ of $\mc$, one obtains $\mathbb{G}(1,i)=(i+1, t-i), 1\leq i\leq t-1$. On the other hand, for any $(a,b)\in \cw_{M}$ such that $1<a\leq t$ and $a+b\leq t+1$, we have the AR-triangle
\[(a-1, b)\to (a-1, b+1)\oplus (a, b-1)\to (a,b)\to \Sigma (a-1,b)
\]
 ending at $(a, b)$ in $\mc_Q$\footnote{If the second coordinate of an object $(a,b)$ is zero, then the object is defined to be the zero object.}. Note that all of $(a-1, b), (a-1,b+1), (a, b-1)$ and $(a,b)$ belong to $~^\perp(\Sigma M)$,
 it induces an AR-triangle in the subfactor $\mc$ and we have $\mathbb{G}(a, b)=(a-1, b)$ by Lemma~\ref{l:AR-triangles}~(2).
We conclude that the indecomposable objects in  $\cw_M/\add M$ form a finite connected AR component of $\mc$, which is the AR quiver of the cluster category of a type $A_{t-1}$ quiver.
\end{proof}

\begin{lemma}~\label{l:AR-tube-infinite}
Indecomposable objects in $\mt_{M,\mc}\backslash \add M$ form a connected AR component $\Gamma_{\bullet}$ of $\mc$.  Moreover, $\Gamma_{\bullet}$ is  a tube of rank $d-t$.
\end{lemma}
\begin{proof}
Let $(a, b)\xrightarrow{f_0}(a,b+1)\xrightarrow{f_1}(a,b+2)\cdots\xrightarrow{f_t}(a, b+t+1)$ be a sequence of irreducible maps of $\mc_Q$ such that $(a, b), (a, b+t+1)\in ~^\perp(\Sigma M)$ and $(a, b+1),\cdots, (a,b+t)\not\in ~^\perp(\Sigma M)$.  Set $f=f_t\circ\cdots \circ f_0$. 
We claim that $\overline{f}$ is an irreducible morphism from $(a, b)$ to $(a, b+t+1)$ in $\mc$.

Without loss of generality, we may assume that $f$ is an $\mathcal{M}$-map.  Suppose that $\overline{f}$ is not irreducible.
There is a factorization of $\overline{f}=\overline{g}\circ \overline{h}$ in $\mc$ such that $\overline{g}:N\to (a,b+t+1)$ is not a retraction and $\overline{h}:(a,b)\to N$ is not a section. Recall that $\mc$ is the additive quotient of $^\perp(\Sigma M)$ with respect to $\add M$. The factorization $\overline{f}=\overline{g}\circ \overline{h}$ lifts to a factorization $f=\hat{g}\circ \hat{h}$ in $\mc_Q$, where $\hat{g}:N\oplus M^{\oplus r}\to (a,b+t+1)$ and $\hat{h}: (a,b)\to N\oplus  M^{\oplus r}$.  We may rewrite $\hat{g}=g_1+g_2$ and $\hat{h}=h_1+h_2$, where $g_1, h_1$ are $\mathcal{M}$-maps and $g_2,h_2$ are $\der$-maps.  Consequently, $f=g_1\circ h_1$ since $f$ is an $\mathcal{M}$-map.
We clearly know that $g_1:N\oplus M^{\oplus r}\to (a,b+t+1)$ is not a retraction and $h_1: (a,b)\to N\oplus M^{\oplus r}$ is not a section in $\mod kQ$.
By the assumption that $(a, b+1),\cdots, (a,b+t)\not\in ~^\perp(\Sigma M)$, we also have $(a,b+i)\not\in \add N\oplus  M^{\oplus r}$ for any $1\leq i\leq t$. 
Now the factorization $f=g_1\circ h_1$ contradicts the fact that each tube in $\mod kQ$ is standard. This completes the proof that $\overline{f}$ is an irreducible morphism in $\mc$.
Dually, if $(a,b)\xrightarrow{g_0} (a+1, b-1)\xrightarrow{g_1}(a+2, b-2)\cdots \xrightarrow{g_t}(a+t+1, b-t-1)$ is a sequence of irreducible maps of $\mc_Q$ such that $(a, b), (a+t+1, b-t-1)\in ~^\perp(\Sigma M)$ and $(a+1,b-1),\cdots, (a+t,b-t)\not\in ~^\perp(\Sigma M)$, then $\overline{g_t\circ\cdots\circ g_0}$ is an irreducible morphism from $(a,b)$ to $(a+t+1, b-t-1)$ in $\mc$. 

According to Lemma~\ref{l:AR-triangles}, we conclude that $\mt_{M,\mc}/\add M$ forms a connected AR component of $\mc$ with infinitely many indecomposable objects. Moreover, $d-t$ is the minimal positive integer such that $\mathbb{G}^{d-t}(d, t+1)\cong (d, t+1)$.  Consequently, $\mt_{M,\mc}/\add M$ is a rank $d-t$ tube of $\mc$ by the shape of AR quiver of a cluster category.
\end{proof}

The following lemma plays a key role in our investigation of dimension vectors of indecomposable $\tau$-rigid modules for cluster-tilted algebras of tame type.
\begin{lemma}~\label{l:key-lemma}
Let $\mc_Q$ be a cluster category of tame type and $M$ an indecomposable regular rigid object of $\mc_Q$. Denote by $\mc=~^\perp(\Sigma M)/\add M$ the subfactor of $\mc_Q$.  
\begin{itemize}
\item[(1)] Let $X$ be an indecomposable object in $\mc$. If $X$ is transjective in~$\mc_Q$, then $X$ is transjective in $\mc$. Consequently, an indecomposable $X\in\mc$ is transjective in $\mc$ if and only if $X$ is transjective in $\mc_Q$ or $X\in \cw_M/\add M$;
\item[(2)] Assume moreover that $q.l.(M)=1$. Then an indecomposable object $X\in \mc$ is regular in $\mc$ if and only if $X$ is regular in $\mc_Q$.
\end{itemize}
\end{lemma}
\begin{proof}
For the first statement, assume that $X$ is regular in $\mc$. By definition, there is an indecomposable object $Y\in \mc$ lying in the same AR component of $X$ such that $\dim_k\End_{\mc}(Y)\geq 2$. We clearly have $\dim_k\End_{\mc_Q}(Y)\geq 2$. In particular, $Y$ is a regular object in $\mc_Q$. Consequently, the AR component of $\mc_Q$ containing $Y$ is a tube $\mt_Y$. If $M\not\in \mt_Y$, then the tube $\mt_Y$ remains a tube in $\mc$, which implies that $X$ is regular in $\mc_Q$, a contradiction.
Now assume that $M\in \mt_Y$.
By Lemma~\ref{l:AR-tube-finite} and ~\ref{l:AR-tube-infinite}, we deduce that $X\in \mt_Y$ which also contradicts to that $X$ is transjective in $\mc_Q$.

If $q.l.(M)=1$, then the above discussion also implies that if $X\not\cong M$ is a regular object of $\mc_Q$ such that $X\in ~^\perp(\Sigma M)$, then $X$ is regular in $\mc$. Now the second statement follows from the first one.
\end{proof}

\begin{theorem}~\label{t:subfactor-tame-type}
Let $\mc_Q$ be a cluster category of tame type $(d_1,d_2,\cdots, d_r)$ and $M$  an indecomposable rigid object of $\mc_Q$. Denote by $\mc=~^\perp(\Sigma M)/\add M$. 
\begin{itemize}
\item[(1)] If $M$ is transjective in $\mc_Q$, then $\mc$ is a cluster category of finite type;
\item[(2)] If $M$ lies in the rank $d_l$ tube with $q.l.(M)=t<d_l-1$, then the cluster category  $\mc$ is the direct sum a cluster category of type $A_{t-1}$ and a cluster category of tame type $(d_1,\cdots, d_{l-1}, d_l-t, d_{l+1},\cdots, d_r)$. In particular, if $q.l.(M)=1$, then $\mc$ is a cluster category of tame type $(d_1,\cdots, d_{l-1}, d_l-1, d_{l+1},\cdots, d_r)$.
\end{itemize}
\end{theorem}
\begin{proof}
The subfactor $\mc$ is a cluster category by Theorem~\ref{t:subfactor-cluster-category}.
Note that $M$ is transjective if and only if $\Sigma M=\tau M$ is transjective.
Then the first statement is a direct consequence of Lemma~\ref{l:hom-vanish-set}, Theorem~\ref{t:subfactor}~(1) and Lemma~\ref{l:AR-shape}~(1).

Let us consider the second statement.  By Lemma~\ref{l:AR-tube-finite},~\ref{l:AR-tube-infinite} and~\ref{l:AR-shape}~(2), we have a decomposition $\mc=\mc_1\oplus\mc_{A_{t-1}}$ as cluster categories, where $\mc_{A_{t-1}}\cong \cw_M/\add M$. It remains to show that $\mc_1$ is a cluster category of tame type $(d_1,\cdots, d_{l-1}, d_l-t, d_{l+1},\cdots, d_r)$. Again by Lemma~\ref{l:AR-tube-infinite} and Lemma~\ref{l:key-lemma}~(1), we deduce that the ranks of non-homogeneous tubes in $\mc$ are precisely $d_1,\cdots, d_{l-1}, d_l-t, d_{l+1}, \cdots, d_r$.
We have to show that $\mc_1$ can not be decomposed as a direct sum of two non-trivial cluster categories.

Otherwise, assume $\mc_1=\mathcal{M}_1\oplus\mathcal{M}_2$, where $\mathcal{M}_1, \mathcal{M}_2$ are cluster categories. Without loss of generality, we may assume that $\mathcal{M}_1$ is a cluster category of tame type. Let $X\in \mathcal{M}_2$ be a transjective object in $\mc$. By Lemma~\ref{l:key-lemma}~(1), we know that $X\in \mc_Q$ is also transjective in $\mc_Q$.  Note that $\mathcal{M}_1$ is a cluster category of tame type and $\Hom_{\mc}(Z, \mathbb{G}(X))=0$ for any $Z\in \mathcal{M}_1$.
In particular, the set $\{Y\in \ind^{\operatorname{rig}}\mc~|~\Hom_{\mc}(Y, \mathbb{G}(X))=0\}$ is infinite.

Now consider the minimal left $\add M$-approximation $X\xrightarrow{\alpha_X}M_X$ of $X$ and form the triangle
\[X\xrightarrow{\alpha_X}M_X\xrightarrow{\beta_X} \mathbb{G}(X)\to \Sigma X
\]
in $\mc_Q$, where $M_X\in \add M$. By the $2$-Calabi-Yau property of $\mc_Q$, we have \[\Hom_{\mc_Q}(M,\Sigma X)=D\Hom_{\mc_Q}(X, \Sigma M)=0.\] In particular, $\beta_X$ is a right $\add M$-approximation of $\mathbb{G}(X)$.
For $Y\in \mc$, applying the functor $\Hom_{\mc_Q}(Y, -)$ to the triangle above, we obtain a long exact sequence
\[\cdots\to\Hom_{\mc_Q}(Y, M_X)\to \Hom_{\mc_Q}(Y, \mathbb{G}(X))\to \Hom_{\mc_Q}(Y, \Sigma X)\to \Hom_{\mc_Q}(Y, \Sigma M_X)\to \cdots
\]
Recall that $Y\in \mc$ implies that $\Hom_{\mc_Q}(Y, \Sigma M_X)=0$. It follows that $\Hom_{\mc_Q}(Y, \Sigma X)\cong \Hom_\mc(Y, \mathbb{G}(X))$. In particular, if $Y\in \ind^{\operatorname{rig}}\mc$ such that $\Hom_{\mc}(Y, \mathbb{G}(X))=0$, then $Y\in \ind^{\operatorname{rig}}\mc_Q$ and $\Hom_{\mc_Q}(Y, \Sigma X)=0$. Consequently, the set $\{Y\in \ind^{\operatorname{rig}} \mc_Q~|~\Hom_{\mc_Q}(Y,\Sigma X)=0\}$ is infinite, which contradicts Lemma~\ref{l:hom-vanish-set}. This completes the proof.
\end{proof}

Let $\mc$ be a cluster category and $T$ a basic cluster-tilting object in $\mc$. If $T$ does not have a regular direct summand, then the cluster-tilted  algebra $\End_{\mc}(T)^{op}$ is called a {\it cluster-concealed algebra}.  We have the following immediately consequence of Theorem~\ref{t:subfactor-tame-type}.
\begin{corollary}
Let $\Gamma$ be a cluster-tilted algebra of tame type. Then $\Gamma$ is a cluster-concealed algebra if and only if for every primitive idempotent $e$, the algebra $\Gamma/\Gamma e\Gamma$ is a cluster-tilted algebra of finite type.
\end{corollary}

\section{Dimension vectors and mutations}~\label{S:dimension-vector}

\subsection{Exchange compatibility}
Let $\mc$ be a cluster category and $(X,X^*)$ an exchange pair of $\mc$. Recall that for the exchange pair $(X,X^*)$, we have two exchange triangles
\[X\xrightarrow{f}B\xrightarrow{g}X^*\to \Sigma X~\text{and}~X^*\xrightarrow{f'}B'\xrightarrow{g'}X\to \Sigma X^*.
\]
An indecomposable object $M\in \mc$ is {\it compatible with the exchange pair $(X,X^*)$}~\cite{BMR09}, if either $X\cong \Sigma M$ or $X^*\cong \Sigma M$, or
\[\dim_k\Hom_\mc(M,X)+\dim_k\Hom_\mc(M,X^*)=\max\{\dim_k\Hom_\mc(M,B), \dim_k\Hom_\mc(M,B')\}.
\]
If $M$ is compatible with every exchange pair $(X,X^*)$ of $\mc$, then $M$ is called {\it exchange compatible}.

Note that $(X,X^*)$ is an exchange pair if and only if $(\Sigma^2X, \Sigma^2X^*)$ is an exchange pair. By definition of compatible and the $2$-Calabi-Yau property, one clearly have the following result~({\it cf.}~\cite{AD})
\begin{lemma}~\label{l:AD}
An indecomposable rigid object $M\in\mc$ is compatible with $(\Sigma^2X, \Sigma^2X^*)$ if and only if either $\Sigma X\cong M$ or $\Sigma X^*\cong M$, or
\[\dim_k\Hom_\mc(X,M)+\dim_k\Hom_\mc(X^*,M)=\max\{\dim_k\Hom_\mc(B,M), \dim_k\Hom_\mc(B',M)\}.
\]
\end{lemma}

Let $(X,X^*)$ be an exchange pair. If both $X$ and $X^*$ are transjective, then we call the pair $(X,X^*)$  a {\it transjective exchange pair} and the corresponding mutation a {\it non-regular mutation}. The following has been proved in~\cite{BM10}.
\begin{lemma}~\label{l:non-regulr-mutation}
Let $Q$ be a finite acyclic quiver and $\mc_Q$ the associated cluster category.
 \begin{itemize}
 \item[(1)] If $(X,X^*)$ is a transjective exchange pair of $\mc_Q$, then each indecomposable rigid object of $\mc_Q$ is compatible with  $(X,X^*)$;
 \item[(2)] Let $T$ be a basic cluster-tilting object of $\mc_Q$. If $T$ has no regular direct summands, then $T$ can be obtained from $kQ$ by a series of non-regular mutations.
 \end{itemize}
\end{lemma}

For cluster categories of tame type, Buan, Marsh and Reiten~\cite{BMR09} obtained the following characterization of exchange compatible rigid objects.
\begin{lemma}~\label{l:compatible-objects}
Let $\mc$ be a cluster category of tame type and $N$ an indecomposable rigid object of $\mc$. Then $N$ is exchange compatible if and only if $\End_\mc(N)\cong k$ if and only if $N$ is transjective or $q.l.(N)\leq t-2$ when $N$ is regular in a tube of rank $t\geq 2$.
\end{lemma}

Let $\mc$ be a cluster category of tame type and $T$ a basic cluster-tilting object of $\mc$. We can always write $T$ as $T=T_{tr}\oplus R$, where $T_{tr}$ is transjective and $R$ is regular. The following result gives a sufficient condition on when two basic cluster-tilting objects can be obtained from each other by non-regular mutations.

\begin{lemma}~\label{l:tame-type-non-regular-mutation}
Let $\mc$ be a cluster category of tame type. Let $T=T_{tr}\oplus R$ and $T'=T'_{tr}\oplus R$ be two basic cluster-tilting objects of $\mc$, where $R$ is regular and $T_{tr}, T_{tr}'$ are transjective. Then the cluster-tilting object $T$ can be obtained from $T'$ by a series of non-regular mutations.
\end{lemma}
\begin{proof}
We proof this result by induction on the number $|R|$. If $|R|=0$, the result follows from Lemma~\ref{l:non-regulr-mutation}~(2) and Lemma~\ref{l:cluster-vs-tilting}~(b) directly. Now assume that the result holds for $|R|=t$ and we consider the case $|R|=t+1$. According to Lemma~\ref{l:property-tame-clustercategory}~(3), there is an indecomposable direct summand $S$ of $R$ such that $q.l.(S)=1$. We may rewrite $T=T_{tr}\oplus R'\oplus S$ and $T'=T'_{tr}\oplus R'\oplus S $.

Let us consider the subfactor $\mc':=~^\perp(\Sigma S)/\add S$ of $\mc$ determined by $S$, which is a cluster category by Theorem~\ref{l:property-tame-clustercategory}. Moreover, there is a one-to-one correspondence between the cluster-tilting objects of $\mc'$ and the cluster-tilting objects of $\mc$ contianing $S$ as a direct summand. In particular, $T_{tr}\oplus R'$ and $T'_{tr}\oplus R'$ are basic cluster-tilting objects of $\mc'$. By the induction, $T_{tr}\oplus R'$ can be obtained from $T'_{tr}\oplus R'$ by a series of non-regular mutations. The sequence of non-regular mutations in $\mc'$ lift to a sequence of non-regular mutations in $\mc$ by Lemma~\ref{l:key-lemma}~(2). This completes the proof.

\end{proof}

\subsection{The behavior of dimension vectors under mutation}
Let $\mc$ be a cluster category and $T=\bigoplus\limits_{i=1}^nT_i$ a basic cluster-tilting object of $\mc$ with indecomposable direct summands $T_1,\cdots, T_n$.
Let $\Gamma=\End_\mc(T)^{op}$ be the cluster-tilted algebra associated to $T$. Recall that for an object $M$, we have a $\Gamma$-module $M_\Gamma=\Hom_\mc(T,M)$. For  the dimension vector $\dimv M_\Gamma$, we have
\[\dimv M_\Gamma:=(\dim_k\Hom_\mc(T_1,M),\cdots, \dim_k\Hom_\mc(T_n,M))\in\Z^n.\]

Let $T'=\mu_{T_k}(T):=\overline{T}\oplus T_k^*$ be the mutation of $T$ at the indecomposable direct summand $T_k$. In particular, $(T_k, T_k^*)$ is an exchange pair. Denote by $\Gamma'=\End_\mc(T')^{op}$ the cluster-tilted algebra of $T'$. Note that $(\Sigma^2T_k,\Sigma^2T_k^*)$ is also an exchange pair.
\begin{proposition}~\label{p:dimension-vector-mutation}
Keep the notations as above. Assume moreover that each indecomposable rigid object of $\mc$ is compatible with $(\Sigma^2T_k, \Sigma^2T_k^*)$. If different indecomposable $\tau$-rigid $\Gamma'$-modules have different dimension vectors, then different indecomposable $\tau$-rigid $\Gamma$-modules have different dimension vectors.
\end{proposition}
\begin{proof}
Let $T_k\xrightarrow{f}B\xrightarrow{g}T_k^*\to \Sigma T_k$ and $T_k^*\xrightarrow{f'}B'\xrightarrow{g'}T_k\to \Sigma T_k^*$ be the exchange triangle associated to the pair $(T_k,T_k^*)$.

Suppose that there are two non-isomorphic indecomposable $\tau$-rigid $\Gamma$-modules $M_\Gamma$ and $N_\Gamma$ such that $\dimv M_\Gamma=\dimv N_\Gamma$, where $M\not\cong N\in \mc\backslash \add\Sigma T$ are indecomposable rigid objects. We first claim that $M\not\cong \Sigma T_k^*$ and $N\not\cong \Sigma T_k^*$. Otherwise, assume that $M\cong \Sigma T_k^*$.  One obtains that $M_\Gamma$ is the simple $\Gamma$-module with $\dimv M_\Gamma=e_k$, where $e_k$ is the $k$-th standard basis of $\Z^n$. Consequently, $\dimv N_\Gamma=e_k$ and $N_\Gamma\cong M_\Gamma$, a contradiction.

Since $M$ and $N$ are compatible with the exchange pair $(\Sigma^2T_k, \Sigma^2T_k^*)$ and $M,N\not\in \add\Sigma T$ , we have
\[\dim_k\Hom_\mc(T_k, M)+\dim_k\Hom_\mc(T_k^*,M)=\max\{\dim_k\Hom_\mc(B,M), \dim_k\Hom_\mc(B',M)\}
\]
and
\[\dim_k\Hom_\mc(T_k, N)+\dim_k\Hom_\mc(T_k^*,N)=\max\{\dim_k\Hom_\mc(B,N), \dim_k\Hom_\mc(B',N)\}
\]
by Lemma~\ref{l:AD}. On the other hand, by $B, B'\in \add \overline{T}$ and $\dimv M_\Gamma=\dimv N_\Gamma$, we get
\[\dim_k\Hom_\mc(B, M)=\dim_k\Hom_\mc(B,N)~\text{and}~\dim_k\Hom_\mc(B',M)=\dim_k\Hom_\mc(B',N).
\]
Consequently, $\dim_k\Hom_\mc(T_k^*,M)=\dim_k\Hom_\mc(T_k^*,N)$, which implies that $\dimv M_{\Gamma'}=\dimv N_{\Gamma'}$, a contradiction. Therefore different indecomposable $\tau$-rigid $\Gamma$-modules have different dimension vectors.
\end{proof}
The proof of the above proposition also implies the following result.
\begin{corollary}~\label{c:dimension-vector-equality}
Keep the notations as above. Let $M,N\in \mc\backslash \add\Sigma T$ be two non-isomorphic indecomposable rigid objects. Suppose that $M, N$ are compatible with the exchange pair $(\Sigma^2T_k,\Sigma^2T_k^*)$ and $\dimv M_\Gamma=\dimv N_\Gamma$. Then $M,N\not\in \add \Sigma T'$ and $\dimv M_{\Gamma'}=\dimv N_{\Gamma'}$.
\end{corollary}

Let $Q$ be a finite acyclic quiver and $\mc_Q$ the associated cluster category.  It is well-known that rigid $kQ$-modules are determined by their dimension vectors ({\it cf.}~\cite{CB}). In this case, the indecomposable $\tau$-rigid $kQ$-modules coincide with the indecomposable rigid $kQ$-modules.
Hence  indecomposable $\tau$-rigid $kQ$-modules are determined by their dimension vectors. Let $T\in \mc_Q$ be a basic cluster-tilting object without regular direct summands. It follows from Lemma~\ref{l:non-regulr-mutation}~(2) that $T$ can be obtained from the cluster-tilting object $kQ$ by a series of non-regular mutations. On the other hand, if $(X,X^*)$ is a transjective exchange pair in $\mc_Q$, then $(\Sigma^2X, \Sigma^2X^*)$ is also transjective. Moreover, each indecomposable rigid object in $\mc_Q$ is compatible with $(\Sigma^2X, \Sigma^2X^*)$ by Lemma~\ref{l:non-regulr-mutation}~(1). Therefore by Proposition~\ref{p:dimension-vector-mutation} and induction on the length of mutations, we have proved the main result of ~\cite{AD}~({\it cf.} also~\cite{G}).
\begin{theorem}~\label{t:cluster-concealed}
Let $\Gamma$ be a cluster-concealed algebra, then indecomposable $\tau$-rigid $\Gamma$-modules are determined by their dimension vectors.
\end{theorem}
\subsection{Dimension vectors for cluster-tilted algebras of tame type}

\begin{lemma}~\label{l:factor-cluster-concealed-algebra}
Let $\mc$ be a cluster category of tame type and $T=T_{tr}\oplus R$ a basic cluster-tilting object of $\mc$, where $T_{tr}$ is transjective and $R$ is regular. Denote by $\Gamma=\End_\mc(T)^{op}$ the cluster-tilted algebra of $T$ and $e_R$ the idempotent of $\Gamma$ associated to $R$. Then the factor algebra $\Gamma':=\Gamma/\Gamma e_R\Gamma$ is a cluster-concealed algebra.
\end{lemma}
\begin{proof}
We prove this result by induction on $|R|$. It is clear that the result holds for $|R|=0$. For $|R|\geq 1$, it follows from Lemma~\ref{l:property-tame-clustercategory}~(3)~that there is an indecomposable direct summand, say $R_1$, of $R$ with $q.l.(R_1)=1$. Let $\mc_1=~^\perp(\Sigma R_1)/\add R_1$ be the subfactor determined by $R_1$. We may rewrite $R$ as $R=R_1\oplus R'$ and then $T_{tr}\oplus R'$ is a basic cluster-tilting object in $\mc_1$ with $|R'|=|R|-1$. Again by Lemma~\ref{l:key-lemma}~(2), we know that $T_{tr}$ is a transjective object in $\mc_1$ and $R'$ is a regular object in $\mc_1$.
 Denote by $A=\End_{\mc_1}(T_{tr}\oplus R')^{op}$ and $A'=A/A e_{R'}A$ the factor algebra of $A$ by the ideal generated by $e_{R'}$, where $e_{R'}$ is the idempotent of $A$ associated to $R'$. By induction, we know that $A'$ is a cluster-concealed algebra and we obtain the desired result by noting that $\Gamma'\cong A'$.
\end{proof}

Now we are in the position to state the main result of this section.
\begin{theorem}~\label{t:main-theorem-dimension-vector}
Let $\Gamma$ be a cluster-tilted algebra of tame type. If $X$ and $Y$ are indecomposable $\tau$-rigid $\Gamma$-modules with $\dimv X=\dimv Y$, then $X$ and $Y$ are isomorphic.
\end{theorem}
\begin{proof}
Let $\mc=\mc_Q$ be a cluster category of tame type and $T$ a basic cluster-tilting object of $\mc$ such that $\Gamma=\End_\mc(T)^{op}$. We may write $T$ as $T=T_{tr}\oplus R$, where $T_{tr}$ is transjective and $R$ is regular. We prove the result by induction on $|R|$.  If $|R|=0$, then $\Gamma$ is a cluster-concealed algebra and the result follows from Theorem~\ref{t:cluster-concealed} directly.

In the following, we assume that $|R|\geq1$.
Recall that each indecomposable $\tau$-rigid $\Gamma$-module can be represented by $M_\Gamma=\Hom_\mc(T, M)$ for an indecomposable rigid object $M\in \mc\backslash \add \Sigma T$. It suffices to prove that if $M$ and $N$ are non-isomorphic indecomposable rigid objects in $\mc\backslash\add\Sigma T$, then $\dimv M_\Gamma\not=\dimv N_\Gamma$.

Let $M$ and $N$ be non-isomorphic indecomposable rigid objects in $\mc\backslash\add\Sigma T$.
Suppose that  $\dimv M_\Gamma=\dimv N_\Gamma$, 
we have either $\Hom_\mc(R, M)=0$~or~$\Hom_\mc(R, M)\neq 0$.

Let us first consider the case $\Hom_\mc(R, M)=0=\Hom_\mc(R, N)$. Denote by $\Gamma'=\Gamma/\Gamma e_R\Gamma$ the factor algebra of $\Gamma$ by the ideal generated by $e_R$, where $e_R$ is the idempotent of $\Gamma$ associated to $R$. It is clear that $M_\Gamma$ and $N_\Gamma$ can be viewed as $\Gamma'$-modules. Moreover, $M_\Gamma$ and $N_\Gamma$ are non-isomorphic indecomposable $\tau$-rigid $\Gamma'$-modules with the same dimension vector. By Lemma~\ref{l:factor-cluster-concealed-algebra}, we know that $\Gamma'$ is a cluster-concealed algebra. Consequently, different indecomposable $\tau$-rigid $\Gamma'$-modules have different dimension vectors by Theorem~\ref{t:cluster-concealed}, a contradiction. 

It remains to consider that $\Hom_\mc(R, M)\neq 0$. Let $R_1$ be an indecomposable direct summand of $R$ such that $\Hom_\mc(R_1, M)\neq 0$. Recall that we have assumed that $\dimv M_\Gamma=\dimv N_\Gamma$, we obtain that $\Hom_\mc(R_1, N)\neq 0$. We separate the remaining part into three cases.

\noindent{\bf Case 1:} both $M$ and $N$ are transjective. According to Lemma~\ref{l:compatible-objects}, $M$ and $N$ are exchange compatible. We may choose a hereditary algebra $H=kQ'$ such that $M$ and $N$ are identified to indecomposable rigid $H$-modules. By Lemma~\ref{l:cluster-vs-tilting}~(b), $T$ can be obtained from the basic cluster-tilting object $H$ by a finite sequence of mutations. Applying Corollary~\ref{c:dimension-vector-equality} repeatedly, we obtain $\dimv \Hom_{\mc_Q}(H, M)=\dimv \Hom_{\mc_Q}(H, N)$. In particular, $M$ and $N$ have the same dimension vector as $H$-modules, a contradiction. 

\noindent{\bf Case 2:} One of $M$ and $N$ is transjective. We may assume that $M$ is transjective and $N$ is regular. Note that, as $R$ and $N$ are regular, there is a positive integer $m$ such that $\Sigma^m R=R$ and $\Sigma^m N=N$. According to Lemma~\ref{l:tame-type-non-regular-mutation}, the cluster-tilting object $\Sigma^m T=\Sigma^m T'\oplus R$ can be obtained from $T$ by a finite sequence of non-regular mutations. Again by Corollary~\ref{c:dimension-vector-equality}, we deduce that $\dimv \Hom_{\mc_Q}(\Sigma^mT, M)=\dimv \Hom_{\mc_Q}(\Sigma^m T, N)$.
For any integer $l$, the cluster-titling object $\Sigma^{lm}T$ can also be obtained from $T$ by a finite sequence of mutations. Therefore, $\dimv \Hom_{\mc_Q}(\Sigma^{lm}T, M)=\dimv \Hom_{\mc_Q}(\Sigma^{lm}T,N)$ for any integer $l$. Consequently, we obtain
\begin{eqnarray*}
\dimv\Hom_{\mc_Q}(T, M)&=&\dimv\Hom_{\mc_Q}(T,N)=\dimv \Hom_{\mc_Q}(\Sigma^{lm}T, \Sigma^{lm}N)\\
&=&\dimv\Hom_{\mc_Q}(\Sigma^{lm}T, N)=\dimv\Hom_{\mc_Q}(\Sigma^{lm}T, M)
\end{eqnarray*}
for any integer $l$. Recall that $|R|\leq |kQ|-2$ and we may choose an indecomposable direct summand $T_i$ of $T$ such that $T_i$ is transjective. Then we have $\dim_k\Hom_{\mc_Q}(T_i, M)=\dim_k\Hom_{\mc_Q}(\Sigma^{lm}T_i, M)$ for any integer $l$. Note that $T_i$ is transjective, which implies that $\Sigma^{sm}T_i\not\cong \Sigma^{tm}T_i$ whenever $s\neq t\in \mathbb{Z}$. In particular, we obtain infinitely many non-isomorphic indecomposable rigid objects $X\in \mc_Q$ such that $\dim_k\Hom_{\mc_Q}(X, M)=\dim_k\Hom_{\mc_Q}(T_i,M)$, which contradicts Lemma~\ref{l:hom-vanish-set}. 

\noindent{\bf Case 3:} both $M$ and $N$ are regular. Recall that we have $\Hom_{\mc_Q}(R_1, M)\neq 0\neq \Hom_{\mc_Q}(R_1, N)$, which implies that
all of $R_1, M$ and $N$ belong to the same tube, say $\mt$ with rank $t$. We claim first that either $q.l.(M)=t-1$ or $q.l.(N)=t-1$. Otherwise, we have $q.l.(M)\leq t-2$ and $q.l.(N)\leq t-2$.
It follows from Lemma~\ref{l:compatible-objects} that $M$ and $N$ are exchange compatible. By Lemma~\ref{l:cluster-vs-tilting}~(b), $T$ can be obtained from $kQ$ by a finite sequence of mutations.
Applying Corollary~\ref{c:dimension-vector-equality}, we deduce that $\dimv M_{kQ}=\dimv N_{kQ}$ as indecomposable rigid $kQ$-modules, a contradiction.

Without loss of generality, we may assume that $q.l.(M)=t-1$.
 By Lemma~\ref{l:property-tame-clustercategory}~(3), $R$ admits an indecomposable direct summand $R_{0}$ lying in $\mt$ with $q.l.(R_{0})=1$. If $\Hom_\mc(R_{0}, M)\neq 0$, then $\Hom_\mc(R_{0},N)\neq 0$ since $\dimv M_\Gamma=\dimv N_\Gamma$. In particular, $M$ and $N$ lie in the ray $\mathrm{R}_{R_{0}}$ or the coray $\mathrm{C}_{\tau^2R_{0}}$. Recall that, as $M$ and $N$ are non-isomorphic indecomposable rigid objects in $\mt$, we deduce that $q.l.(N)\leq t-2$. By Lemma~\ref{l:property-tame-clustercategory} (4), we obtain $\dim_k\Hom_\mc(R_{0}, M)=2$ and $\dim_k\Hom_\mc(R_{0},N)=1$, which contradicts to $\dimv M_\Gamma=\dimv N_\Gamma$. It follows that $\Hom_\mc(R_{0}, M)=0=\Hom_\mc(R_{0},N)$.  Let $\Gamma''=\Gamma/\Gamma e_{R_0}\Gamma$ be the factor algebra of $\Gamma$ by the ideal generated by $e_{R_0}$, where $e_{R_0}$ is the idempotent of $\Gamma$ corresponding to $R_0$. The $\Gamma$-module $M_\Gamma$ and $N_\Gamma$ can be viewed as $\Gamma''$-modules. Moreover, $M_\Gamma$ and $N_\Gamma$ are non-isomorphic indecomposable $\tau$-rigid $\Gamma''$-modules with the same dimension vector.

  Denote by $R=R'\oplus R_0$ and consider the subfactor $\mc':=~^\perp(\Sigma R_0)/\add R_0$ of $\mc$ determined by $R_0$. By Theorem~\ref{t:subfactor-tame-type}~(2), $\mc'$ is also a cluster category of tame type. Moreover, $T_{tr}\oplus R'$ is a basic cluster-tilting object of $\mc'$ and $\Gamma''=\End_{\mc'}(T_{tr}\oplus R')^{op}$.
By Lemma~\ref{l:key-lemma}~(2), we deduce that $R'$ is precisely the regular direct summand of $T_{tr}\oplus R'$. Note that $|R'|=|R|-1$, by induction, we deduce that different indecomposable $\tau$-rigid $\Gamma''$-modules have different dimension vectors. However, $M_\Gamma$ and $N_\Gamma$ are non-isomorphic indecomposable $\tau$-rigid $\Gamma''$-modules with the same dimension vector, a contradiction. This finishes the proof.

\end{proof}

\section{Denominator conjectures for cluster algebras of tame type}~\label{S:denominator}
\subsection{Recollection of denominator vectors} In this subsection, we recall the definitions of cluster algebras and denominators. For simplicity, we restrict ourselves to the skew-symmetric and coefficient-free cases. For the general situation, we refer to ~\cite{FZ02,FZ03-2}.

Fix an integer $n$. For an integer $a$, we set $[a]_+=\max\{a,0\}$. Let $\mf$ be the field of fractions of the ring of polynomials in $n$ indeterminates with coefficients in $\mathbb{Q}$. A {\it seed} in $\mf$ is a pair $(B, \mathbf{x})$ consisting of a skew-symmetric matrix $B\in M_n(\mathbb{Z})$ and a free generating set $\mathbf{x}=\{x_1,\cdots, x_n\}$ of the field $\mf$. The matrix $B$ is the {\it exchange matrix} and $\mathbf{x}$ is the {\it cluster} of the seed $(B, \mathbf{x})$. Elements of the cluster $\mathbf{x}$ are {\it cluster variables} of the seed $(B, \mathbf{x})$. The mutation of matrix may be extended to the mutation of seeds. Namely, for any $1\leq k\leq n$, the {\it seed mutation} of $(B, \mathbf{x})$ in the direction $k$ transforms $(B, \mathbf{x})$ into a new seed $\mathbf{\mu_k}(B, \mathbf{x})=(\mu_k(B), \mathbf{x}')$, where the cluster $\mathbf{x}'=\{x_1',\cdots, x_n'\}$ is given by $x_j'=x_j$ for $j\neq k$ and $x_k'\in \mf$ is determined by the {\it exchange relation}
\[x_k'x_k=\prod_{i=1}^nx_i^{[b_{ik}]_+}+\prod_{i=1}^nx_i^{[-b_{ik}]_+}.
\]
The {\it cluster algebra} $\ma(B)=\ma(B,\mathbf{x})$ is the subalgebra of $\mf$ generated by all the cluster variables which can be obtained from the initial seed $(B, \mathbf{x})$ by iterated mutations.
As we have seen in Section~\ref{s:cluster-structure}, there is a one-to-one correspondence between skew-symmetric matrices and quivers without loops nor $2$-cycles. For a finite quiver $Q$ without loops nor $2$-cycles, we denote by $\ma(Q):=\ma(B(Q))$ the cluster algebra associated to the skew-symmetric matrix $B(Q)$.
A seed $(B, \mathbf{x})$ is {\it acyclic} if the quiver associated to $B$ is acyclic.
A cluster algebra is called {\it acyclic} if it admits an acyclic seed.

Let $Q$ be a finite quiver without loops nor $2$-cycles and $\ma(Q)$ the corresponding cluster algebra. A {\it cluster monomial} of $\ma(Q)$ is a monomial in cluster variables all of which belong to the same cluster.
 Denote by $Q_0=\{1,\cdots, n\}$ the vertex set of $Q$. For any given cluster $Y=\{y_1,\cdots, y_n\}$ of $\ma(Q)$, it follows from the Laurent phenomenon~\cite{FZ02} that every cluster monomial $x$ of $\ma(Q)$ can be written uniquely as
\[x=\frac{f(y_1,\cdots, y_n)}{\prod\limits_{i=1}^ny_i^{d_i}},
\]
where $d_1,\cdots, d_n\in \mathbb{Z}$ and $f(y_1,\cdots, y_n)$ is a polynomial in $y_1,\cdots, y_n$ which can not be divisible by any $y_i, 1\leq i\leq n$.
\begin{definition}
The integer vector $\operatorname{den}(x):=(d_1,\cdots, d_n)$ is called the denominator vector of $x$ with respect to the cluster $Y$.
\end{definition}
Inspired by Lusztig's parameterization of canonical bases in quantum groups, Fomin and Zelevinsky~\cite{FZ03-2} proposed the following {\bf denominator conjecture}.
\begin{conjecture}
Different cluster monomials have different denominator vectors with respect to a given cluster.
\end{conjecture}

To our best knowledge, the denominator conjecture has not yet been verified even for the cluster algebras of finite type.
Thus, in order to establish the denominator conjecture, one may first to consider a weak version of the denominator conjecture. Namely,
different cluster variables have different denominators.
 In the following, we will verify the weak denominator conjecture for cluster algebras of tame type with respect to any initial seeds.

\subsection{Denominators for cluster variable of tame type}~\label{s:denominator-tame-type}
In this subsection, we fix a connected extended Dynkin quiver $Q$. Denote by $Q_0=\{1,\cdots, n\}$ the vertex set of $Q$. Let $\mc_Q$ be the cluster category of $Q$ and $\ma(Q)$ the cluster algebra associated to $Q$.
 There is a bijection between the cluster variables of $\ma(Q)$ and the indecomposable rigid objects of $\mc_Q$~\cite{CK06}. Moreover, the clusters of $\ma(Q)$ correspond to the basic cluster-tilting objects of $\mc_Q$. For a given cluster $Y=\{y_1,\cdots, y_n\}$ of $\ma(Q)$, denote by $\Sigma T=\bigoplus_{i=1}^n(\Sigma T_i)\in \mc_Q$ the corresponding basic cluster-tilting object, where $\Sigma T_1,\cdots, \Sigma T_n$ are indecomposable rigid objects corresponding to $y_1,\cdots, y_n$ respectively. Let $\Gamma=\End_{\mc_Q}(T)^{op}$ be the cluster-tilted algebra. We have a bijection between the indecomposable $\tau$-rigid $\Gamma$-modules and the indecomposable rigid objects in $\mc\backslash\add \Sigma T$ by Lemma~\ref{l:bijection-tau-rigid-and-rigid}.
Consequently, there is a one-to-one correspondence between the indecomposable $\tau$-rigid $\Gamma$-modules and the non-initial cluster variables of $\ma(Q)$. In the following, for an indecomposable rigid object $M\in \mc_Q$, we write $x_M$ for the corresponding cluster variable and denote by $\operatorname{den}(x_M)$ the denominator vector of $x_M$. The following result gave an explicit relation between the denominator vector of $x_M$ and the dimension vector of $M_\Gamma$ ({\it cf.}~\cite{BM10}).
\begin{theorem}~\label{t:affine-denominator-dimension}
Let $M_\Gamma$ be an indecomposable $\tau$-rigid $\Gamma$-module and $\operatorname{den}(x_M)=(d_1,\cdots, d_n)$ the denominator vector of the cluster variable $x_M$. Then we have
\[d_i=\begin{cases}\dim_k\Hom_{\mc_Q}(T_i, M)-1& ~\text{if there is a tube of rank $t\geq 2$ containing $T_i$ and}\\
&~\text{  $M$ such that $q.l.(T_i)=t-1$ and $M\not\in \cw_{\tau T_i}$;}\\
\dim_k\Hom_{\mc_Q}(T_i,M)& otherwise.\end{cases}
\]
\end{theorem}
The following result is a consequence of Lemma~\ref{l:compute-morphism}~({\it cf.} also~\cite{BM10}).
\begin{lemma}~\label{l:dimenon-affine-type}
Let $M$ and $N$ be indecomposable rigid objects in $\mc_Q\backslash\add\Sigma T$. Suppose that $M$ and $N$ lie in the same tube of rank $t$ and $q.l.(N)=t-1$. Then we have
\[\dim_k\Hom_{\mc_Q}(N, M)=\begin{cases}2& M\not\in \cw_{\tau N};\\
0& M\in \cw_{\tau N}.\end{cases}
\]
\end{lemma}
Combining Theorem~\ref{t:affine-denominator-dimension} and Lemma~\ref{l:dimenon-affine-type}, we clearly know that the denominator vectors of non-initial cluster variables of $\ma(Q)$ are nonnegative.

\subsection{Weak denominator conjecture of tame type}
We begin with the following lemma which plays an important role in the proof of the weak denominator conjecture of cluster algebra of tame type.
\begin{lemma}~\label{l:mutation-transjective-vs-regular}
Let $\mc_Q$ be a cluster category of tame type. Let $T=\overline{T}\oplus T_k$ and $\mu_{T_k}(T)=\overline{T}\oplus T_k^*$ be two basic cluster-tilting objects of $\mc_Q$ such that both $T_k$ and $T_k^*$ are transjective. Denote by $T_i$ an indecomposable regular direct summand of $T$. Let $M$ be an indecomposable transjective object of $\mc_Q$ and $N$ an indecomposable rigid object of $\mc_Q$ which lies in the same tube as $T_i$. Assume that $M,N\not\in \add\Sigma T$ and for any indecomposable direct summand $T_j$ of $T$
\[\dim_k\Hom_{\mc_Q}(T_j, M)=\begin{cases}\dim_k\Hom_{\mc_Q}(T_j,N)& j\neq i;\\ \dim_k\Hom_{\mc_Q}(T_i,N)-1&j=i,\end{cases}
\]
then $M,N\not\in \add\Sigma \mu_{T_k}(T)$ and $\dim_k\Hom_{\mc_Q}(T_k^*,M)=\dim_k\Hom_k(T_k^*,N).$
\end{lemma}
\begin{proof}
Let $T_k^*\xrightarrow{f}B\xrightarrow{g}T_k\to \Sigma T_k^*$ and $T_k\xrightarrow{f'}B'\xrightarrow{g'}T_k^*\to \Sigma T_k$ be the exchange triangles associated to the exchange pair $(T_k, T_k^*)$. It is clear that $(\Sigma^2 T_k, \Sigma^2T_k^*)$ is also a transjective exchange pair and the corresponding exchange triangles are precisely $\Sigma^2T_k^*\xrightarrow{\Sigma^2f}\Sigma^2B\xrightarrow{\Sigma^2g}\Sigma^2T_k\to \Sigma^3 T_k^*$ and $\Sigma^2T_k\xrightarrow{\Sigma^2f'}\Sigma^2B'\xrightarrow{\Sigma^2g'}\Sigma^2T_k^*\to \Sigma^3 T_k$. Moreover, $M$ and $N$ are compatible with the exchange pair $(\Sigma^2T_k, \Sigma^2T_k^*)$ by Lemma~\ref{l:non-regulr-mutation}~(1).
Note that, as $T_k^*$ is transjective and $N$ is regular, we clearly have $N\not\cong \Sigma T_k^*$. On the other hand, $N\not\in \add \Sigma T$ implies that $N\not\cong \Sigma T_k$. Consequently, by Lemma~\ref{l:AD}, we obtain the following equality
\begin{eqnarray}~\label{f:dimension-equality}
&&\dim_k\Hom_{\mc_Q}(T_k, N)+\dim_k\Hom_{\mc_Q}(T_k^*,N)\\
&=&
\max\{\dim_k\Hom_{\mc_Q}(B, N), \dim_k\Hom_{\mc_Q}(B',N)\}.\notag
\end{eqnarray}
Since $(\Sigma^2T_k, \Sigma^2T_k^*)$ is a transjective exchange pair, we may choose a hereditary algebra $H=kQ'$ which is derived equivalent to $kQ$,  such that $\Sigma^2T_k$ and $\Sigma^2T_k^*$ are identified to preinjective $H$-modules.
Hence one of the exchange triangles associated to $(\Sigma^2T_k, \Sigma^2T_k^*)$, say
 \[\Sigma^2T_k^*\xrightarrow{\Sigma^2f}\Sigma^2B\xrightarrow{\Sigma^2g}\Sigma^2T_k\to \Sigma^3 T_k^*,
 \]
 is induced by a short exact sequence of $H$-modules~({\it cf.}~\cite{BMRRT}).
 Consequently, $\Sigma^2B$ is also a preinjective $H$-module. Note that, as $N$ is regular, $N$ is identified to a regular $H$-module. Applying the functor $\Hom_H(N,-)$ to the short exact sequence $\Sigma^2 T_k^*\rightarrowtail \Sigma^2B\twoheadrightarrow \Sigma^2T_k$ of preinjective $H$-modules and by Lemma~\ref{l:compute-morphism}, we have
 \[0\to\Hom_{\mc_Q}(N,\Sigma^2T_k^*)\to \Hom_{\mc_Q}(N,\Sigma^2B)\to \Hom_{\mc_Q}(N,\Sigma^2T_k)\to 0.
 \]
 The $2$-Calabi-Yau property of $\mc_Q$ implies the following short exact sequence
 \[0\to\Hom_{\mc_Q}(T_k, N)\to \Hom_{\mc_Q}(B,N)\to \Hom_{\mc_Q}(T_k^*,N)\to 0.
 \]
In other words, we have
\begin{eqnarray}~\label{f:dimension-B}
\dim_k\Hom_{\mc_Q}(B,N)=\dim_k\Hom_{\mc_Q}(T_k,N)+\dim_k\Hom_{\mc_Q}(T_k^*,N).
\end{eqnarray}
Consequently, $\dim_k\Hom_{\mc_Q}(B,N)\geq \dim_k\Hom_{\mc_Q}(B',N)$ by equality~(\ref{f:dimension-equality}).

We claim that $M\not\cong \Sigma T_k^*$. Otherwise, one has $\dim_k\Hom_{\mc_Q}(T_j,M)=\begin{cases}0& j\neq k;\\ 1& j=k.\end{cases}$
By the assumption $\dim_k\Hom_{\mc_Q}(T_j, M)=\begin{cases}\dim_k\Hom_{\mc_Q}(T_j,N)& j\neq i;\\ \dim_k\Hom_{\mc_Q}(T_i,N)-1&j=i,\end{cases}$, we deduce that
\[\dim_k\Hom_{\mc_Q}(T_j,N)=\begin{cases}1& j=k;\\ 1&j=i;\\ 0& \text{else}.\end{cases}
\]
Note that $B\in\add \overline{T}$ is trasnjective, which implies that $B$ does not admit $T_i$ as a direct summand. In particular,  $\Hom_{\mc_Q}(B, N)=0$ and consequently $\dim_k\Hom_{\mc_Q}(T_k^*,N)=-1$ by ~(\ref{f:dimension-B}), a contradiction. Hence, $M\not\cong \Sigma T_k^*$. By the condition that $M,N\not\in \add \Sigma T$, we have proved that $M,N\not\in \add \Sigma \mu_{T_k}(T)$.

Recall that $M$ is transjective, which implies that $M$ is also compatible with the exchange pair $(\Sigma^2T_k,\Sigma^2T_k^*)$. On the other hand, we also have $M\not\cong \Sigma T_k$ and $M\not\cong \Sigma T_k^*$. By Lemma~\ref{l:AD}, we obtain
\begin{eqnarray}~\label{f:dimension-MB}
&&\dim_k\Hom_{\mc_Q}(T_k,M)+\dim_k\Hom_{\mc_Q}(T_k^*,M)\\
&=&\max\{\dim_k\Hom_{\mc_Q}(B, M),\dim_k\Hom_{\mc_Q}(B',M)\}.\notag
\end{eqnarray}
Notice that $B$ does not admit $T_i$ as a direct summand.
By the assumption, we have 
\[\dim_k\Hom_{\mc_Q}(B,M)=\dim_k\Hom_{\mc_Q}(B,N)~\text{and}~\dim_k\Hom_{\mc_Q}(B',M)\leq \dim_k\Hom_{\mc_Q}(B',N).\] Recall that we have also shown that \[\dim_k\Hom_{\mc_Q}(B,N)\geq \dim_k\Hom_{\mc_Q}(B',N).\] Putting all of these together, we get $\dim_k\Hom_{\mc_Q}(B',M)\leq \dim_k\Hom_{\mc_Q}(B,M)$ and
\begin{eqnarray}\label{f:dimension-M}
\dim_k\Hom_{\mc_Q}(T_k,M)+\dim_k\Hom_{\mc_Q}(T_k^*,M)=\dim_k\Hom_{\mc_Q}(B,M)
\end{eqnarray}
by equality~(\ref{f:dimension-MB}).
We conclude that $\dim_k\Hom_{\mc_Q}(T_k^*,M)=\dim_k\Hom_{\mc_Q}(T_k^*,N)$ by (\ref{f:dimension-B}) and (\ref{f:dimension-M}).
\end{proof}
\begin{theorem}~\label{t:denominator-conjecture-tame-type}
Let $Q$ be a connected extended Dynkin quiver and $\ma(Q)$ the cluster algebra associated to $Q$. Then different cluster variables of $\ma(Q)$ have different denominators with respect to any given initial seed.

\end{theorem}
\begin{proof}
Let $Q_0=\{1,2,\cdots,n\}$ be the vertex set of $Q$ and $\mc=\mc_Q$ the cluster category associated to $Q$. We fix a cluster $Y=\{y_1,\cdots,y_n\}$ of $\ma(Q)$ as initial seed.
Recall that there is a bijection between the cluster variables of $\ma(Q)$ and the indecomposable rigid objects of $\mc$.
For any indecomposable rigid object $M$ of $\mc$, we denote by $x_M$ the corresponding cluster variable. Denote by $\Sigma T=\bigoplus_{i=1}^n(\Sigma T_i)$ the basic cluster-tilting objects corresponding to $Y$, where $\Sigma T_1,\cdots, \Sigma T_n$ are indecomposable rigid objects correspond to $y_1,\cdots, y_n$ respectively. Namely, we have $x_{\Sigma T_i}=y_i, 1\leq i\leq n$. Let $\Gamma=\End_{\mc_Q}(T)^{op}$ be the cluster-tilted algebra associated to $T$.
According to Theorem~\ref{t:affine-denominator-dimension}, for any indecomposable rigid object $M\in \mc\backslash \add\Sigma T$, we clearly have $\operatorname{den}(x_M)\neq \operatorname{den}(x_{\Sigma T_i})$ for any $1\leq i\leq n$. Thus it remains to prove that for any non-isomorphic indecomposable rigid objects $M, N\in \mc\backslash\add\Sigma T$, we have $\operatorname{den}(x_M)\neq \operatorname{den}(x_N)$.

In the following, for an indecomposable rigid object $M\in \mc\backslash\add\Sigma T$,  $M$ is said to satisfy the property $(P)$,
 if $T$ admits an indecomposable direct summand $T_i$ such that
\begin{itemize}
\item[(i)] $T_i$ and $M$ belong to the same tube of rank $t\geq 2$ and $q.l.(T_i)=t-1$;
\item[(ii)] $M\not\in \cw_{\Sigma T_i}$.
\end{itemize}
If both $M$ and $N$ do not satisfy the property $(P)$, then $\operatorname{den}(x_M)=\dimv M_\Gamma$ and $\operatorname{den}(x_N)=\dimv N_\Gamma$ by Theorem~\ref{t:affine-denominator-dimension}. Consequently, $\operatorname{den}(x_M)\neq \operatorname{den}(x_N)$ by Theorem~\ref{t:main-theorem-dimension-vector}. In particular, if $M$ and $N$ are transjective, then $\operatorname{den}(x_M)\neq \operatorname{den}(x_N)$.
In the following, we may assume that $N$ is regular and there is an indecomposable direct summand $T_i$ of $T$ such that $T_i\in \mt_N$ with $q.l.(T_i)=t-1$ and $N\not\in \cw_{\Sigma T_i}$, where $t$ is the rank of the tube $\mt_N$. It is not hard to see that $T_i$ is the unique indecomposable direct summand of $T$ such that $T_i\in \mt_N$ with $q.l.(T_i)=t-1$. In this situation, we have $\operatorname{den}(x_N)=\dimv N_\Gamma-e_i$ by Theorem~\ref{t:affine-denominator-dimension}, where $e_1,\cdots, e_n$ is the standard basis of $\mathbb{Z}^n$.
We separate the remaining part into three cases.

\noindent{\bf Case 1:} $M\in\mt_N$.
If $M\not\in \cw_{\Sigma T_i}$, we have $\operatorname{den}(x_M)=\dimv M_\Gamma-e_i$ by Theorem~\ref{t:affine-denominator-dimension}. It follows from Theorem~\ref{t:main-theorem-dimension-vector} that $\operatorname{den}(x_M)\neq \operatorname{den}(x_N)$.
Now assume that $M\in \cw_{\Sigma T_i}$. According to Lemma~\ref{l:dimenon-affine-type}, we have $\dim_k\Hom_{\mc}(T_i, M)=0$ and $\dim_k\Hom_{\mc}(T_i, N)=2$. Again by Theorem~\ref{t:affine-denominator-dimension}, we deduce that $\operatorname{den}(x_M)\neq \operatorname{den}(x_N)$ in this case.

\noindent{\bf Case 2:} $M$ and $N$ belong to different tubes. Note that we have $\dim_k\Hom_{\mc}(T_i, N)=2$.
On the other hand, $\dim_k\Hom_{\mc}(T_i,M)=0$ by the fact that there are no non-zero morphisms between different tubes. Consequently, $\operatorname{den}(x_M)\neq \operatorname{den}(x_N)$ by Theorem~\ref{t:affine-denominator-dimension}.

\noindent{\bf Case 3:} $M$ is transjective. Suppose that we have $\operatorname{den}(x_M)=\operatorname{den}(x_N)$ in this case. By Theorem~\ref{t:affine-denominator-dimension}, we obtain
\[\dim_k\Hom_{\mc}(T_j,M)=\begin{cases}\dim_k\Hom_\mc(T_j,N)& j\neq i;\\
\dim_k\Hom_\mc(T_i,N)-1& j=i.\end{cases}
\]
In particular, $M$ and $N$ satisfy the condition of Lemma~\ref{l:mutation-transjective-vs-regular}.

We rewrite $T$ as $T=T_{tr}\oplus R$, where $T_{tr}$ is transjective and $R$ is regular. In particular, there exists a positive integer $m$ such that $\Sigma^m R=R$ and $\Sigma^mN=N$. By Lemma~\ref{l:tame-type-non-regular-mutation}, $\Sigma^m T$ can be obtained from $T$ by a series of non-regular mutations. Let $T_j$ be an indecomposable direct summand of $T_{tr}$. Applying Lemma~\ref{l:mutation-transjective-vs-regular} repeatedly, we get
\[\dim_k\Hom_\mc(\Sigma^mT_j,M)=\dim_k\Hom_\mc(\Sigma^mT_j,N).
\]
Note that for any integer $l$, we also have $\Sigma^{lm}R=R$ and $\Sigma^{lm}N=N$. Similarly, one obtains
\[\dim_k\Hom_\mc(\Sigma^{lm}T_j,M)=\dim_k\Hom_\mc(\Sigma^{lm}T_j,N).
\]
Consequently, for any integer $l$, we have
\[\dim_k\Hom_\mc(\Sigma^{lm}T_j,M)=\dim_k\Hom_\mc(\Sigma^{lm}T_j,N)=\dim_k\Hom_\mc(T_j,N)=\dim_k\Hom_\mc(T_j,M).
\]
Note that $\mc$ is a cluster category of tame type and $T_j$ is a transjective object. It is clear that $\Sigma^{tm}T_j\not\cong \Sigma^{sm}T_j$ whenever $s\neq t$. In particular, there are infinitely non-isomorphic indecomposable rigid objects $X$ such that $\dim_k\Hom_\mc(X,M)=\dim_k\Hom_\mc(T_j,M)$, which contradicts Lemma~\ref{l:hom-vanish-set}. Consequently, $\operatorname{den}(x_M)\neq \operatorname{den}(x_N)$ in this case. This completes the proof.
\end{proof}

\end{document}